\documentclass[11pt, a4paper]{article}
\usepackage[latin1]{inputenc}
\usepackage[T1]{fontenc}
\usepackage{amsmath}
\usepackage{amsfonts}
\usepackage{amssymb}
\usepackage{amsthm}
\usepackage{stmaryrd}
\usepackage{graphicx}
\usepackage[affil-it]{authblk}
\usepackage{enumitem}
\usepackage[normalem]{ulem}
\usepackage[mathscr]{euscript}
\usepackage{dsfont}
\usepackage{float}
\usepackage{times}
\usepackage{xcolor}
\usepackage{xparse}
\usepackage[hang]{footmisc}
\usepackage[colorlinks,linkcolor=blue,citecolor=blue,urlcolor=blue]{hyperref}
\usepackage[capitalise]{cleveref}
\usepackage[left=3cm, right=3cm, bottom=3cm, top=3cm]{geometry}

\newtheorem{theo}{Theorem}[section]
\newtheorem{prop}[theo]{Proposition}
\newtheorem{defi}[theo]{Definition}

\newtheorem{lem}[theo]{Lemma}
\newtheorem{exe}[theo]{exe}

\newtheorem{Rq}[theo]{Remark}

\newcommand{\un}{\underline}
\newcommand{\tronc}{\mathfrak{t}}
\newcommand{\TV}{TV} 

\newcommand{\Zee}{X}
\newcommand{\mutilde}{\tilde{\mu}}

\newcommand{\KK}{{\mathfrak S}}
\newcommand{\kmu}{\mathbb{K}_\mu}
\newcommand{\knu}{\mathbb{K}_\nu}

\newcommand{\tcr}{\textcolor{red}}

\newcommand{\ER}{\mathbb{R}}
\newcommand{\cim}{\partial}

\newcommand{\PE}{\mathbb{P}}

\newcommand{\ES}{\mathbb{E}}
\newcommand{\XX}{\xi}

\newcommand{\MM}{\mathcal{E}}

\newcommand{\Aa}{{\cal A}}

\numberwithin{equation}{section}
\newcommand{\vertiii}[1]{{\left\vert\kern-0.25ex\left\vert\kern-0.25ex\left\vert #1 \right\vert\kern-0.25ex\right\vert\kern-0.25ex\right\vert}}
\newcommand{\rom}[1]{(\textup{\uppercase\expandafter{\romannumeral#1}})}

\makeatletter
\newcommand{\substackal}[1]{%
  \vcenter{%
    \Let@ \restore@math@cr \default@tag
    \baselineskip\fontdimen10 \scriptfont\tw@
    \advance\baselineskip\fontdimen12 \scriptfont\tw@
    \lineskip\thr@@\fontdimen8 \scriptfont\thr@@
    \lineskiplimit\lineskip
    \ialign{\hfil$\m@th\scriptstyle##$&$\m@th\scriptstyle{}##$\hfil\crcr
      #1\crcr
    }%
  }%
}
\makeatother

\newcommand\blfootnote[1]{%
  \begingroup
  \renewcommand\thefootnote{}\footnote{#1}%
  \addtocounter{footnote}{-1}%
  \endgroup
}

\newlist{hzero}{enumerate}{1}
\setlist[hzero]{label=$\mathbf{(A_0)}$:, ref=$\mathbf{(A_0)}$, wide, labelwidth=!, itemindent=!, itemindent=!, labelindent=0pt}

\newlist{hun}{enumerate}{1}
\setlist[hun]{label=$\mathbf{(A_1)}$:, ref=$\mathbf{(A_1)}$, wide, labelwidth=!, itemindent=!, labelindent=0pt}

\newlist{hvdeux}{enumerate}{1}
\setlist[hvdeux]{label=$\mathbf{(A_V)}$:, ref=$\mathbf{(A_V)}$, wide, labelwidth=!, itemindent=!, labelindent=0pt}

\newlist{hdeux}{enumerate}{1}
\setlist[hdeux]{label=$\mathbf{(A_2)}$:, ref=$\mathbf{(A_2)}$, wide, labelwidth=!, itemindent=!, labelindent=0pt}

\newlist{htrois}{enumerate}{1}
\setlist[htrois]{label=$\mathbf{(A_3)}$:, ref=$\mathbf{(A_3)}$, wide, labelwidth=!, itemindent=!, labelindent=0pt}

\newlist{hquatre}{enumerate}{1}
\setlist[hquatre]{label=$\mathbf{(A_4)}$:, ref=$\mathbf{(A_4)}$, wide, labelwidth=!, itemindent=!, labelindent=0pt}

\newlist{hmv}{enumerate}{1}
\setlist[hmv]{label=$\mathbf{(H_{MV})}$:, ref=$\mathbf{(H_{MV})}$, wide, labelwidth=!, itemindent=!, labelindent=0pt}

\newlist{htronc}{enumerate}{1}
\setlist[htronc]{label=$\mathbf{(H_{\tronc})}$:, ref=$\mathbf{(H_{\tronc})}$, wide, labelwidth=!, itemindent=!, labelindent=0pt}

\newlist{pzero}{enumerate}{1}
\setlist[pzero]{label=$\mathbf{(P_0)}$:, ref=$\mathbf{(P_0)}$, wide, labelwidth=!, itemindent=!, labelindent=0pt}

\newlist{pun}{enumerate}{1}
\setlist[pun]{label=$\mathbf{(P_1)}$:, ref=$\mathbf{(P_1)}$, wide, labelwidth=!, itemindent=!, labelindent=0pt}

\newlist{pdeux}{enumerate}{1}
\setlist[pdeux]{label=$\mathbf{(P_2)}$:, ref=$\mathbf{(P_2)}$, wide, labelwidth=!, itemindent=!, labelindent=0pt}

\begin{document}
\title{Quasi-Stationary Distributions of Interacting Dynamical Systems and their approximation}
\author[1]{{Mohamed Alfaki \textsc{Aboubacrine Assadek}}}
\author[2]{Fabien \textsc{Panloup}}
\affil[1,2]{Univ Angers, CNRS, LAREMA, SFR MATHSTIC, F-49000 Angers, France}
\date{\today}

\newgeometry{top=0cm, bottom=1.5cm}

\maketitle
\thispagestyle{empty}

\blfootnote{Email addresses:  \href{mailto:mohamedalfaki.agaboubacrineassadeck@univ-angers.fr}{mohamedalfaki.agaboubacrineassadeck@univ-angers.fr}, \href{mailto:fabien.panloup@univ-angers.fr}{fabien.panloup@univ-angers.fr}, 
 \\[1ex] \scriptsize
  This work is supported by  the Centre Henri Lebesgue ANR-11-LABX-0020-01 and by ANR RAWABRANCH.
}

\vspace{-3em}

\begin{abstract}
  \noindent
 In \cite{BCP}, the authors built and studied an algorithm based on the (self)-interaction of a dynamics with its occupation measure to approximate Quasi-Stationary Distributions (QSD) of general Markov chains conditioned to stay in a compact set. In this paper, we propose to tackle the case of McKean-Vlasov-type dynamics, \emph{i.e.} of dynamics interacting with their marginal distribution (conditioned to not be killed). In this non-linear setting, we are able to exhibit some conditions which guarantee that weak limits of these sequences of random measures are QSDs of the given dynamics. We also prove tightness results in the non-compact case.\\

\noindent  These general conditions are then applied to Euler schemes of McKean-Vlasov SDEs and in the compact case, the behavior of these QSDs when the step $h$ goes to $0$ is investigated. Our results also allow to consider some examples in the non-compact case and some new tightness criterions are also provided in this setting. Finally, we illustrate our theoretical results with several simulations.\\[1ex]
  \noindent\textbf{MSC2010:} Primary 65C20, 60B12, 60J99; Secondary: 34A34, 34E10.\\
  \noindent\textbf{Keywords:} Quasi-Stationary Distribution, McKean-Vlasov, Self-interacting dynamics.
\end{abstract}

{
\small
\hypersetup{hidelinks}
\setcounter{tocdepth}{2}
\tableofcontents
}

\newpage

\restoregeometry

\section{Introduction}

The study of the long-time behavior of stochastic processes which are absorbed at a given stopping time (usually exit time)  is now a widely studied problem in the Markovian setting.  Under adequate conditions, Markovian dynamics conditioned to stay alive converge in distribution to the so-called quasi-stationary distribution which can be viewed as an equilibrium related to normalized absorbed dynamics. We refer for instance to the seminal paper \cite{CV14} for general results which ensure convergence at an exponential rate and to  \cite{cloez_esaim} for a survey on recent contributions on the topic. For practical applications, conditioned dynamics attracted much attention for modelling of population dynamics (see \emph{e.g.} \cite{cattiaux_collet,chazottes_collet} ) and for speeding up optimization dynamics (see \emph{e.g.} \cite{lebris_lelievre} and the references therein).\\

\noindent  The question of its approximation thus becomes an important problem and a series of papers investigated mainly two ways of approximations procedures: the Fleming-Viot algorithm (see \emph{e.g.} \cite{BHM00,CT13,MM00,V11}) where the principle is to simulate a particle system induced by the dynamics and to manage the conditioning by making reborn the killed particles with an offspring being through the empirical measure of the alive particles. Another approach initiated by  \cite{BC15} (see also \cite{BGZ,BCP,benaim_champagnat_villemonais,wang_roberts}) is based on  self-interacting dynamics, \emph{i.e.}, where the distribution of the reborn particles is given by the occupation measure of the process. \\

In this paper, we address quasistationary distributions for dynamics with interaction with their distribution that we call \textit{McKean-Vlasov-type} dynamics in the sequel, and especially to their approximation. Our terminology is certainly related to the so-called McKean-Vlasov SDEs whose ``standard'' convergence to equilibrium, \emph{i.e.} to its ``standard'' invariant distribution has been widely studied in the last years (among others, see \emph{e.g.} \cite{guillin_wu} and the references therein). For absorbed McKean-Vlasov-type dynamics, the references seem to be more rare. The topic is adressed in \cite{tough_nolen} where the authors study the Fleming-Viot dynamics fo McKean-Vlasov SDEs and exhibit some conditions which ensure that for the exit time of a compact set, the  Fleming-Viot algorithm converges to the set of QSDs of such dynamics. At this stage, it is important to remark that in order to make sense, absorbed McKean-Vlasov SDEs are certainly considered by interacting with the conditioned distribution: 
\begin{equation}\label{eq:mckeansde}
 \begin{cases} &d\xi_t=b(\xi_t,\mu_t) dt+\sigma(\xi_t,\mu_t) dW_t\\
&\mu_t={\cal L}(\xi_t|\tau>t),
\end{cases}
\end{equation}
where $\tau$ denotes the exit time and $b:\ER^d\times{\cal P}(\ER^d)\rightarrow\ER^d$ and $\sigma:\ER^d\times{\cal P}(\ER^d)\rightarrow\mathbb{M}_{d,d}$ (the space of $d\times d$ real matrices) denote the drift and diffusion coefficients respectively (whose regularity will be precised later).\\

\noindent  Initially motivated by the objective to consider the self-interacting approach related to continuous-time dynamics such as \eqref{eq:mckeansde}, our work focuses in fact to general discrete-time dynamics with interaction\footnote{As mentioned later, the continuous-time setting, especially the \textit{hard-killing} problem, contains additional difficulties that we leave for a future paper. Yet, we tackle the continuous-time by a weak approach considering QSDs of McKean-Vlasov SDEs as limits of QSDs of related Euler schemes.}. In this sense, it is mainly a generalization to interacting dynamics of \cite{BCP}.  Our prototypical example is the Euler-Maruyama scheme of \eqref{eq:mckeansde} but our results may apply to other types of interacting models. Our main contributions are the following:
\begin{itemize}
\item{} For a general Polish space $\MM$,  for a family of kernels of submarkovian dynamics $(K_{\mu,\partial})_{\mu\in{\cal P}(\MM)}$ and for a stopping time $\tau$, we begin by providing a characterization of QSDs of such interacting dynamics as the fixed point of dynamics with redistribution (whose markovian kernel is denoted by $\kmu$ in the sequel).
\item{} We build and study an associated self-interacting sequence of random probabilities denoted by $(\mu_n)_{n\ge1}$ (where $\mu_n\in{\cal P}(\MM)$ is the occupation measure of the dynamics). 
\begin{itemize}
\item{} We provide Lyapunov criterions which ensure almost sure tightness of $(\mu_n)_{n\ge1}$.
\item{} Denoting by $(\tilde{\mu}^{(n)}_t)_{t\ge0}$ a function with values in ${\cal P}(\MM)$ built as an adequate interpolation of the sequence $(\mu_k)_{k\ge n}$, we show that every weak limit of the sequence\footnote{We also show that this sequence is tight for the topology of uniform convergence on compact intervals on the space ${\cal C}([0,+\infty),{\cal P}(\MM))$ of continuous functions with values in the space of probabilities on $\MM$ (endowed with the topology of weak convergence).} $(\tilde{\mu}^{(n)})_{n\ge1}$ is a solution to an ordinary differential equation (ode) on the space of probabilities.
\item{} We exhibit conditions which ensure that this limiting ode is stationary which in turn implies that the limit points of $(\mu_n)$ are almost surely QSDs. This last result (corresponding to \Cref{theo:discret}$(ii)$  is probably  the part of the paper, where the non-linearity of the dynamics generates the most difficulties. To obtain this property, we show uniform convergence properties of the limiting ode but without uniqueness assumptions (see \Cref{prop:convuniformemustar}). Note that this result is obtained under Assumption \ref{condhquatre} which is a typical assumption for ensuring uniqueness of QSDs in the linear case.
 \item{} We apply our general result to the Euler scheme of a McKean-Vlasov SDE by proving, under ellipticity assumptions and boundedness of the \textit{flat derivatives} of the coefficients, that the procedure converges\footnote{By ``converges'', we abusively mean that the sequence is tight and every weak limit is a QSD. This certainly leads to convergence when uniqueness holds but we do not address this (probably difficult) question in this paper.}. In the compact case, we are able to show that when the step $h$ goes to $0$, tightness holds (under adequate conditions) and limiting distributions (when $h\rightarrow0$) are QSDs of the McKean-Vlasov SDE. In the non-compact case, we are able to provide a result for Euler schemes of Lévy-driven SDEs \textit{coming down from infinity}\footnote{This terminology is classical in the literature of QSDs and roughly says that the dynamics can come back in a compact set in a bounded time. Such an assumption is usually required to get uniqueness of QSDs in the Markovian case.} when the driving noise has polynomial decrease at infinity. Curiously, our class of examples may contain stable driven SDEs but not the case where the driving noise is the Brownian motion (precisely, we are not able to prove the above assumption \ref{condhquatre} in this case). In the Brownian case, we still have a non trivial tightness result under a Lyapunov assumption which is for instance satisfied by Ornstein-Uhlenbeck processes with interaction (see \Cref{lem:hvdeux} combined with \Cref{lem:lowerupper}).
\end{itemize}
\item{} We end by some numerical illustrations in the same case as the one considered in \cite[Example 2.15]{tough_nolen} and show that our simulations allow to retrieve the (explicit) QSDs related to this specific case.
\end{itemize}
To end this description of our main results, let us remark that this contribution could be probably easily adapted to the continuous-time case for soft-killing SDEs but considering the continuous-time for the hard-killing problem (as a generalization of \cite{benaim_champagnat_villemonais}) seems to generate much more difficulties. In some sense, this is why we limit ourselves to a ``weak'' result which approximates QSDs of McKean-Vlasov SDEs by letting the time go to $\infty$ and then the step go to $0$.\\

\noindent \textbf{Outline.} The paper is organized as follows:  In Section \ref{setandmainres}, we detail the framework and the main assumptions, and state our main results about the convergence of the algorithm towards the QSD. We provide examples and numerical illustrations in \cref{sec:examples}. \\

The sequel of the paper is devoted to the proofs of the results. In \cref{sec:proof1}, we provide preliminary properties and show that the limiting dynamics of the (continuous-time extension of the) algorithm is a solution to a non-linear ode on the space of probabilities. In \cref{sec:proof2}, we exhibit long-time convergence properties of this non-linear ode. As mentioned before, this in some sense in this part, that the non-linearity implies the most difficulties. In \cref{subsecseptun}, we gather the results of the two previous sections to prove \cref{theo:discret}. In \cref{sec:proof3}, we prove \cref{prop:tightnessbis} which provides sufficient Lyapunov-type conditions to ensure tightness of the sequence of empirical measures. Finally, in \cref{sec:proof4}, we focus on proofs and additional results related to Euler schemes of McKean-Vlasov SDEs. We provide explicit tightness criterions and prove \cref{prop:model1} and \cref{prop:model2}.

\section{Setting and Main Results}\label{setandmainres}
\subsection{Notation and Setting}

\textbf{Notation.} Let ${\MM}$ be a Polish space equipped with its Borel $\sigma$-field ${\cal B}(\MM).$ Throughout, we let  ${\cal P}(\MM)$  denote the space of (Borel) probabilities over $\MM$ equipped with the topology of weak convergence. For all $\mu \in {\cal P}(\MM)$ and $f \in L^1(\mu)$ we write $\mu(f)$ for $\int_{\MM} f d\mu.$  Recall that $\mu_n \rightarrow \mu$ in ${\cal P}(\MM)$ provided $\mu_n(f) \rightarrow\mu(f)$ for all $f \in C_b(\MM)$ (denoting the bounded continuous functions on ${\MM}$). The space of signed measures on $\MM$ is denoted ${\cal M}_s(\MM)$ and the total variation norm is denoted by $\|\,.\|_{TV}$.    \\

\noindent A {\em  sub-Markovian kernel} on $\MM$ is a map $Q : \MM \times {\cal B}(\MM) \mapsto [0,1]$ such that for all $x \in \MM,$ $A \rightarrow Q(x, A)$ is a measure and for all $A \in {\cal B}(\MM), x \rightarrow Q(x,A)$ is measurable.  If furthermore $Q(x, \MM) = 1$ for all $x \in M,$ then $Q$ is called a {\em Markov (or Markovian) kernel}. For every bounded measurable function $f : \MM \mapsto \mathbb{R},$  and  $\mu \in {\cal P}(\MM),$  we let $Qf$ and $\mu Q$ respectively denote the map and measure defined by
$$Qf(x) = \int_{\MM} f(y) Q(x,dy)\quad \textnormal{and}\quad \mu Q(\cdot) = \int_{\MM} \mu(dx)Q(x,\cdot).$$
A (sub)-Markovian kernel is Feller if for all $f \in C_b(\MM)$, $Qf  \in C_b(\MM)$. For an integer $d\ge1$, the scalar product and the Euclidean norm on $\ER^d$ are denoted by $\langle.,.\rangle$ and $|.|$ whereas $\mathbb{M}_{d,d}$ is the space of $d\times d$ matrices. For $p>0$, we set  ${\cal P}_p(\ER^d)=\{\mu\in {\cal P}(\ER^d),\mu(|\,.\,|^p)<+\infty\}$. A map $U:{\cal P}_p(\ER^d)\rightarrow\ER$ is said to be of class ${\cal C}^1$ if there exists a jointly continuous and locally bounded map $\frac{\delta U}{\delta m}:{\cal P}_p(\ER^d)\times \ER^d\rightarrow \ER$ satisfying: $\sup_{m\in {\cal P}_p(\ER^d)}|\frac{\delta U}{\delta m}(m,x)|\lesssim(1+|x|^p)$ and for all $m_0$ and $m_1\in{\cal P}(\ER^d)$,
$$ U(m_1)-U(m_0)=\int_0^1 \int_{\ER^d} \frac{\delta U}{\delta m}(t m_1+(1-t) m_0,y) (m_1-m_0)(dy) dt.$$
In this case, we set
$$\left\|\frac{\delta U}{\delta m}(m,.)\right\|_\infty=\sup_{x\in\ER^d}\left|\frac{\delta U}{\delta m}(m,x)\right|\quad \textnormal{and}\quad 
\left[\frac{\delta U}{\delta m}(m,.)\right]_1=\sup_{x,y\in\ER^d, x\neq y} \frac{\left|\frac{\delta U}{\delta m}(m,y)-\frac{\delta U}{\delta m}(m,x)\right|}{|y-x|}.$$

\medskip

\noindent \textbf{Setting.} Let $ ({K}_{\mu,\partial})_{\mu}$ denote a family of    subMarkovian kernels on $\MM$ and let $(Y_n)_{n\ge0}$ denote the sequence defined on $\MM\cup \{\partial\}$ by: $Y_0\sim \nu_0$
\begin{equation}\label{eq:defYn}
Y_{n+1}=\begin{cases} 
Z_{n+1} & \textnormal{if $\tau>n+1$}\\
\partial &  \textnormal{if $\tau\le n+1$},
\end{cases}
\end{equation}
where 
\begin{itemize}
\item{} $\tau$ is a $({\cal F}_n)_{n\ge0}$-stopping time defined by
$$ \tau:=\inf\{n\ge 1, \Theta_n=1\},$$
where $(\Theta_n)_{n\ge1}$ is a sequence of Bernoulli random variables such that 
$$\PE(\Theta_{n+1}=1|{\cal F}_n, Y_n\in\MM)=\delta_{\nu_n}(Y_n)\quad\textnormal{with}\quad \delta_\mu(x)=1-K_{\mu,\partial}{\bf 1}(x),$$
$({\cal F}_n)_{n\ge0}$ is defined by ${\cal F}_0=\sigma(Y_0)$ and for $n\ge1$,
$${\cal F}_n=\sigma(Y_0,\ldots,Y_n, \Theta_1,\ldots,\Theta_n),$$
and 
$$\nu_n={\cal L}(Y_n|\tau>n).$$
\item{} $(Z_n)_{n\ge1}$ is a sequence of random variables with values in $\MM$ such that 
$${\cal L}(Z_{n+1}|{\cal F}_n, Y_n\in\MM)=K_{\nu_n,\partial}(Y_n,.).$$
\end{itemize}
\begin{Rq} The above definition implies in an underlying way that $\PE(\tau>n)>0$. Nevertheless, the definition is still consistent if $\PE(\tau>n)=0$ since in this case $Y_n=\partial$ almost surely.
\end{Rq}
\noindent \textbf{Model 1:} (Euler Scheme of a  Mc-Kean Vlasov diffusion killed at an exit time of a relatively compact set). Assume that $D$ is a bounded open set of $\ER^d$. Let $h>0$.   Let $(\zeta_t)_{t\ge0}$ denote a Brownian motion on $\ER^d$ or more generally a given centered L\'evy process. We also set
$$\Delta \zeta_{s,t}=\zeta_t-\zeta_s, \quad 0\le s\le t.$$
Let $\partial$ denote a cemetery point. Set ${\MM}=\bar{D}$. We define $(Y_n)_{n\ge0}$ by $Y_0\sim\mu_0\in{\cal P}(\MM)$ and
$$\forall n\ge 0, \quad Y_{n+1}=\begin{cases} Y_n+h b(Y_n,\nu_n)+\sigma(Y_n,\nu_n)\Delta \zeta_{nh,(n+1)h}&\textnormal{while $\tau>n+1$,}\\
\partial &\textnormal{if $\tau\le n+1$,}
\end{cases}
$$
where $b:\MM\times {\cal P}(\MM)$ and $\sigma:\MM\times{\cal P}(\MM)\rightarrow\mathbb{M}_{d,d}$ denote (at least jointly continuous) functions,
$$\tau:=\inf\{n\in\mathbb{N}, Y_n+h b(Y_n,\nu_n)+\sigma(Y_n,\nu_n)\Delta \zeta_{nh,(n+1)h}\in \MM^c\},$$
and 
$$\nu_n={\cal L}(Y_n|\tau>n).$$
\begin{Rq} \label{rq: model1}$\rhd$ If $\mu_0$ has support in $D$ and $\sigma$ is not degenerated (typically uniformly elliptic), $\tau$ is almost surely equal to the time where the scheme leaves $D$ (this will allow us to work on the compact space $\MM=\bar{D}$ in the sequel for the study of the exit time of the open set $D$).\\

\noindent  $\rhd$ Note that in this example, $K_{\mu,\partial}$ (denoted $K_{\mu,\partial}^{(h)}$) is defined by: for every measurable function $f:\MM\rightarrow\ER$, 
\begin{equation}\label{eq:formegeneralenoyau}
 K_{\mu,\partial}^{(h)}f(x)=\ES[f(\Upsilon_x^h){\bf 1}_{\Upsilon_x^h\in D}]\quad\textnormal{with}\quad 
\Upsilon_x^h=x+h b(x,\mu)+\sigma(x,\mu)\Delta \zeta_{0,h}.
\end{equation}
$\rhd$ This Euler scheme is related to the McKean-Vlasov diffusion:
$$
\begin{cases}
dY_t=b(Y_t,\nu_t) dt+\sigma(Y_t,\nu_t) dW_t\\
\nu_t={\cal L}(Y_t|\tau>t)
\end{cases}
$$
where $\tau:=\inf\{t\ge0, Y_t\in D^c\}$.
\end{Rq}
\noindent \textbf{Model 2:}  (Euler Scheme of a  McKean-Vlasov diffusion killed at an exit time of an unbounded set)  Excepted the undboundness, this is exactly the same model. Nevertheless, we distinguish them in order to emphasize the specific difficulties involved by the non-compactness. Furthermore, this numbering will correspond to our two types of  examples in the sequel. The exit time of an unbounded set appears naturally in applications. Think for instance to population dynamics where the exit time corresponds to the extinction of the population: in this case, $\bar{D}=[0,+\infty)$. 

%
%
%
\begin{defi} A probability $\nu\in{\cal P}(\MM)$ is called a Quasi-Stationary Distribution (QSD) for $(K_{\mu,\partial})_{\mu\in{\cal P}(\MM)}$ if $Y_0\sim\nu$ implies that
\begin{equation}\label{def:nbyn}
\forall n\in\mathbb{N},\quad \PE(Y_n\in .,\tau>n)=\nu(.)\PE(\tau>n).
\end{equation}
\end{defi}
By an induction, one easily checks that in this case, $\tau$ has an exponential distribution with parameter $\lambda=\nu K_{\nu,\partial} (\MM)$.  In particular, $\PE_\nu(\tau>n)$ is positive for every $n$ and the definition then implies that
$$\forall n\in\mathbb{N},\quad {\cal L}(Y_n|\tau>n)=\nu.$$
\begin{Rq}
We thus retrieve the classical properties of the Markov setting. Nevertheless, it is important to keep in mind that this property also involves that under the ``quasi-stationary regime'',
the dynamics of the killed process is markovian (since at each time $n$, the interaction distribution is also equal to $\nu$). This is similar to the (classical) stationary regime of dynamics with interaction where one retrieves the Markov property when the process starts from its invariant distribution.
\end{Rq}

Before going further, let us introduce two first assumptions:
\begin{hzero}\hypertarget{hzero}{}
\item{}\label{condhzero} There exists $\ell \ge 1$ such $\underset{x\in\MM, \mu\in {\cal P}(\MM)}{\sup} K_{\mu,\partial}^\ell {\bf 1} (x)<1$.
\end{hzero}
where, here and in the sequel, we write ${\bf 1}$ instead of ${\bf 1}_{\MM}$ (so that in particular, $K_{\mu,\partial}^\ell {\bf 1} (x)=K_{\mu,\partial}^\ell (x,\MM)$).
Before going further, let us remark that (see \cref{sec:proofsappendix}),
\begin{lem}\label{lemao} If $\MM$ is compact and if for every $k\ge 1$, $(x,\mu)\mapsto K_{\mu,\partial}^k {\bf 1} (x)$ is continuous on $\MM\times{\cal P}(\MM)$, then \ref{condhzero} holds true if for every $(x,\mu)\in \MM\times{\cal P}(\MM)$, there exists $\ell\ge 1$ such that $K_{\mu,\partial}^\ell {\bf 1}(x)<1$.
\end{lem}

\begin{hun}\hypertarget{hun}{}
\item{}\label{condhun} For every $\mu\in{\cal P}(\MM)$, $K_{\mu,\partial}$ is Feller and for all $x\in\MM$, $\mu\mapsto \mu K_{\mu,\partial}$ is continuous for the topology of weak convergence on ${\cal P}(\MM)$.
\end{hun}

These assumptions will be useful for a first characterization of the QSD based on the \textit{dynamics with renewal}: let $(X_n^\mu)$ denote the Markov chain with kernel 
\begin{equation}\label{eq:kmudef}
\kmu(x,.)=K_{\mu,\partial}(x,.)+\delta_\mu(x)\mu(.).
\end{equation}
This means that when the particle is killed, it is regenerated through the distribution $\mu$.
Set
\begin{equation}\label{eq:Amu}
\Aa_\mu(x,.)=\sum_{n\ge0} K_{\mu,\partial}^n(x,.),
\end{equation}
Under \ref{condhzero} and \ref{condhun}, it can be shown (see \Cref{lem:basics}) that $A_\mu(x,.)$ is well-defined for every $(x,\mu)\in\MM\times{\cal P}(\MM)$ and  that
$$\Pi_\mu=\frac{\mu \Aa_{\mu}}{\mu \Aa_{\mu}{\bf 1}}$$
is the unique invariant distribution for $(X_n^\mu)_{n\ge1}$. We introduce a Lyapunov-type assumption which will be used for existence of QSD in the non-compact case (and tightness in the sequel).\\

 
 \begin{hvdeux}\hypertarget{hvdeux}{}
\item{}\label{condhvdeux} For every $\varepsilon>0$, there exists $\beta_\varepsilon\in\ER$  and an inf-compact\footnote{A function $V:\MM\rightarrow\ER$ is inf-compact if for all $r>0$,  $K_r=\{x\in \MM, V(x)\le r\}$ is compact.} function $V_\varepsilon:\MM\rightarrow\ER_+$ such that
$$\forall \,\mu\in{\cal P}(\MM), \quad K_{\mu,\partial} V_\varepsilon\le \varepsilon V_\varepsilon+\beta_\varepsilon.$$
\end{hvdeux}
\begin{Rq}
We refer to \Cref{lem:hvdeux} for concrete conditions on dynamics to ensure \ref{condhvdeux}.
\end{Rq}
\noindent We have the following property:
\begin{lem}\label{lem:characQSD}
Assume \ref{condhzero} and \ref{condhun}.\\

\noindent (i) $\mu^\star$ is a QSD for $(K_{\mu,\partial})_{\mu\in{\cal P}(\MM)}$ if and only if 
$$\mu^\star=\Pi_{\mu^\star}.$$
(ii) Furthermore, if $\MM$ is compact or, if \ref{condhvdeux} holds and $\inf_{\mu \in{\cal P}(\MM)} \mu K_{\mu,\partial}{\bf 1}>0$, then   existence holds for the QSD related to $(K_{\mu,\partial})_{\mu\in{\cal P}(\MM)}$.

\end{lem}
The fact that $\mu^\star$ can be viewed as a fixed point of $\mu\mapsto \Pi_\mu$ will correspond to the natural target of our algorithm. We mean that our algorithm is naturally connected to the map $\mu\mapsto \Pi_\mu$ and this is why, we chose to recall this characterization in the above lemma. Note that $\inf_{\mu \in{\cal P}(\MM)} \mu K_{\mu,\partial}{\bf 1}>0$ is equivalent to $\sup_{\mu} \mu\delta_\mu<1$, see \Cref{rq:weights} for comments on this assumption (also remark that if there exists $x\in\MM$ and $\mu$  such that $K_{\mu,\partial}{\bf 1}(x)=0$, then existence also holds, $\delta_x$ being a trivial QSD in this case).\\

\noindent Finally, let us mention that we will not consider uniqueness problem in this paper, which seems to be a hard question in this non-linear setting. Note that in \Cref{sec:simu},
we consider a very simple example where uniqueness is (already) not true.\\

\noindent \textbf{Self-Attractive dynamics}. To approximate the QSDs of $(K_{\mu,\partial})_{\mu \in {\cal P}(\MM)}$, one introduces the self-attractive process $(X_n,\mu_n)_{n\ge0}$ recursively defined on  $\MM \times {\cal P}(\MM)$ by: 
$$X_0=x\in{\MM},\quad  \mu_0=\delta_x$$ and for every $n\ge0$,
\begin{equation}\label{dynamicsxn}
\PE(X_{n+1}\in .|{\cal F}_n)= \mathbb{K}_{\mu_n}(X_n,.)
\end{equation}
where ${\cal F}_n=\sigma(X_0,\ldots,X_n)$ and
$$\mu_n=\frac{1}{H_n}\sum_{k=1}^n \eta_k  \delta_{X_{k-1}}\quad\textnormal{with}\quad H_n=\sum_{k=1}^n\eta_k. $$
is a related weighted occupation measure: the $\eta_k$ are supposed to be positive numbers.  Throughout the paper, we also assume that $(\gamma_n)_{n\ge1}$ defined by $\gamma_n=:=\frac{\eta_n}{H_n}$ is \textit{nonincreasing} and that 
\begin{equation}\label{hyp:pas}
H_n\xrightarrow{n\rightarrow+\infty}+\infty \quad\textnormal{and}\quad  \gamma_n\overset{n\rightarrow+\infty}{=} o\left(\frac{1}{\log n}\right).
\end{equation}
Note that the most natural choice is $\eta_n=1$ which leads to $\gamma_n=\frac{1}{n}$ (see \Cref{rq:weights} for additional remarks on, the weights). The sequence $(\mu_n)_{n\ge1}$ satisfies the following recursive formula:
\begin{equation}\label{eq:munrecursive}
\mu_{n+1}=\mu_n (1-\gamma_{n+1})+\gamma_{n+1}\delta_{X_{n+1}}.
\end{equation}
The aim is to prove that the limit points of $(\mu_n)_{n\ge0}$ are QSDs of $(K_{\mu,\partial})_{\mu\in{\cal P}(\MM)}$.  Following \cite{BCP}, we artificially write as follows:
$$\mu_{n+1}=\mu_n+\gamma_{n+1} F(\mu_n)+\gamma_{n+1}\varepsilon_{n+1},$$
with 
$$ F(\mu)=-\mu+\Pi_\mu\quad \textnormal{and}\quad \varepsilon_{n+1}=\delta_{X_{n+1}}-\Pi_{\mu_n}.$$
With this form, a QSD is a zero of $F$ and the bet is to show that $(\mu_n)$ behaves asymptotically as a discretization of the ODE (this means that $\varepsilon_n$ is asymptotically a noise) and that the ODE 
$\dot{\mu}=F(\mu)$ 
has sufficiently nice properties which ensure convergence properties towards its zeros. We will see that the non-linearity will induce many difficulties for this point.\\

\noindent To state our main result, we will need to introduce an interpolated continuous version of $(\mu_n)$. Let $(t_k)_{k\ge0}$ be the increasing sequence of positive numbers defined by: $t_0=0$ and for all $k\ge 1$,
$t_k=\sum_{m=1}^k \gamma_m$.  Extend the above sequence $(\mu_n)_{n\ge1}$ to $\mathbb{N}\cup\{0\}$ by assuming that $\mu_0$ is a given probability on $\MM$. Then, we define $(\mutilde_t)_{t\ge 0}$ by: $\tilde{\mu}_0=\mu_0$ and  for all $k\ge 0$,
 \begin{equation}\label{eq:hnpluss}
    \forall s \in [0,\gamma_{k+1}], \qquad {\mutilde}_{t_{k}+s} = \mu_k + \frac{s}{\gamma_{k+1}}(\mu_{k+1}-\mu_k),
  \end{equation}
$(\mutilde_t)_{{t\ge0}}$ is thus almost surely an element of  the space  ${\cal C}(\ER_+,{\cal P}(\MM))$ of continuous processes on $\ER_+$ with values in ${\cal P}(\MM)$ (for the weak topology) and we consider the related shifted sequence $(\mutilde^{(n)})_{n\ge1}$ defined by:
\begin{equation}\label{processmutilde}
\forall t\ge0, \quad \mutilde^{(n)}_t=\mutilde_{t_n+t}.
\end{equation}
We now introduce  additional assumptions:

\begin{hdeux}\hypertarget{hdeux}{}
\item{}\label{condhdeux} $\mu\mapsto K_{\mu,\partial}(x,.)$ is Lipschitz for the TV-norm, uniformly in $x\in\MM$: there exists $C>0$ such that 
$$\sup_{x\in\MM} \|K_{\mu,\partial}(x,.)-K_{\nu,\partial}(x,.)\|_{TV}\le C \|\mu-\nu\|_{TV}.$$
\end{hdeux}
\begin{htrois}\hypertarget{htrois}{}
\item{}\label{condhtrois} There exist $\ell \in \mathbb{N}$, $\ell \ge1$, $\varepsilon>0$ such that  there exists $\Psi\in{\cal P}({\MM})$ such that for every $\mu\in{\cal P}(\MM)$,
$$ K_{\mu,\partial}^\ell(.,dy)\ge \varepsilon \Psi(dy).$$
\end{htrois}

\noindent The last assumption is related to the same probability $\Psi$ than in \ref{condhtrois}.\\
\begin{hquatre}\hypertarget{hquatre}{}
\item{}\label{condhquatre} There exists $c>0$ such that for every $k\in\mathbb{N}$, for every $\mu_1,\ldots, \mu_k\in{\cal P}(\MM)$, for every $n_1,\ldots, n_k\in\mathbb{N}$,
$$\frac{\Psi(K_{\mu_1,\partial}\circ\ldots\circ K_{\mu_k,\partial} {\bf 1})}{K_{\mu_1,\partial}\circ\ldots\circ K_{\mu_k,\partial} {\bf 1}}\ge c.$$
\end{hquatre}
This assumption  is discussed  in \Cref{rem:hquatre} below and in \Cref{lem:lowerupper}.


\begin{theo}\label{theo:discret} Assume \ref{condhzero}, \ref{condhun}, \ref{condhdeux} and \ref{condhtrois}. Assume  that  $(\mu_n)_{n\ge1}$ is $a.s.$ tight on ${\cal P}(\MM)$ and that $(\gamma_n)$ satisfies \eqref{hyp:pas}. Then,
\begin{itemize}
\item[(i)] The sequence $(\mutilde^{(n)})_{n\ge1}$ is almost surely tight on ${\cal C}(\ER_+,{\cal P}(\MM))$ and every weak limit of $(\mutilde^{(n)})_{n\ge1}$ is a solution of the ODE
\begin{equation}
\label{eq:ODEE}
\frac{d\nu_t}{dt}=-\nu_t+\Pi_{\nu_t}.
\end{equation}
\item[(ii)] If furthermore  \ref{condhquatre} holds, then every weak limit of $(\mutilde^{(n)})_{n\ge1}$ is stationary which implies that every weak limit of $(\mu_n)_{n\ge1}$ is a QSD $\mu^\star$ related to $(K_{\mu,\partial})_{\mu\in{\cal P}(\MM)}$.
\end{itemize}
\end{theo}
\begin{Rq} \label{rem:hquatre} $\rhd$ In the non-compact case, the above result requires $a.s.$- tightness  of $(\mu_n)_{n\ge1}$. A criterion is given in \Cref{prop:tightnessbis} below.\\

\noindent $\rhd$ As quoted in \Cref{rem:poissonization}, a Poissonization argument could avoid to assume \ref{condhtrois} in the first part of \Cref{theo:discret}$(i)$ and to replace \ref{condhtrois} by ${\cal A}_\mu(.,dy)\ge \varepsilon \Psi(dy)$ in \Cref{theo:discret}$(ii)$. Nevertheless, as \ref{condhquatre} requires in practice a much more constraining assumption, we chose to skip this refinement (which would had added technicalities).\\

$\rhd$  For (linear) Markov dynamics, $i.e.$ when the kernel does not depend on $\mu$,  \ref{condhquatre} is  an assumption which is used to ensure uniqueness of the QSD. Unfortunately, this is not the case in this non linear setting (in the same spirit as uniform ellipticity may ensure uniqueness of invariant distributions for linear diffusions but not -in general- for McKean-Vlasov SDEs). Nevertheless, it provides uniformity in the convergence to equilibrium of the ODE \eqref{eq:ODEE}, which in turn helps  us to derive stationarity of the weak limits of $(\mutilde^{(n)})_{n\ge1}$ and  lead to the identification of the limit. 

\noindent The main example where \ref{condhquatre} is satisfied is the ``lower/upper-bounded'' case which is given in \Cref{lem:lowerupper} below.
\end{Rq}
\begin{lem}\label{lem:lowerupper} If  there exists a probability $\Psi$ on $\MM$ and some positive $c_1$ and $c_2$ such that, for every non-negative measurable function $f$ 
\begin{equation}
\forall (x,\mu)\in\MM\times {\cal P}(\MM),\quad c_1 \Psi(f)\le K_{\mu,\partial}f(x)\le c_2 \Psi(f),
\end{equation}
then \ref{condhquatre} holds true.
\end{lem}
\begin{proof} Under this assumption, 
$$c_1\Psi(K_{\mu_2,\partial}\ldots\circ K_{\mu_k,\partial} {\bf 1})\le K_{\mu_1,\partial}\circ\ldots\circ K_{\mu_k,\partial} {\bf 1}\le c_2\Psi(K_{\mu_2,\partial}\circ\ldots\circ K_{\mu_k,\partial} {\bf 1}).$$
Hence, integrating with respect to $\Psi$,
$$\Psi(K_{\mu_1,\partial}\circ\ldots\circ K_{\mu_k,\partial} {\bf 1})\ge c_1\Psi(K_{\mu_2,\partial}\ldots\circ K_{\mu_k,\partial} {\bf 1})\ge \frac{c_1}{c_2}K_{\mu_1,\partial}\circ\ldots\circ K_{\mu_k,\partial} {\bf 1}.$$
\end{proof}
\begin{theo} \label{prop:tightnessbis} Assume \ref{condhvdeux} and $\sup_{\mu\in{\cal P}(\MM)} \|\delta_\mu\|_\infty<1$. Assume that $\sum \gamma_n^2<+\infty$. Then, there exists an inf-compact function ${\cal V}:\MM\rightarrow\ER_+$ such that
$$\sup_{n\ge1} \mu_n({\cal V})<+\infty\quad a.s.$$
In particular, $(\mu_n)$ is $a.s.$ tight.
\end{theo}
\begin{Rq} \label{rq:weights} The assumption  on $\sum \gamma_n^2$ is satisfied in many cases. For instance, if $\eta_n= n^{\alpha}$ with $\alpha\in (-1,\infty)$ then 
$$ H_n\sim \frac{n^{1+\alpha}}{1+\alpha}\quad \textnormal{and}\quad \gamma_n\sim \frac{1}{(1+\alpha) n}.$$
One can check that it is still true if   $\eta_n= e^{n^\alpha}$ with $\alpha<1/2$. The condition on $\delta_\mu$ says that, uniformly in $\mu$ and $x\in\MM$, the probability to die in one step is upper-bounded. In the compact case, this is not constraining for our discrete-time dynamics (it may be for continuous-time \textit{hard-killing} problems). In the non-compact, it generates constraints when the drift is ``superlinear'' (see \ref{subsec:unboundedsuperlinear}). For instance, if $D=(0,+\infty)$ and $b(x)=-x^2$, the probability to die related to the Euler scheme goes to $1$ when $x\rightarrow+\infty$. This is due to a classical unstability of the Euler scheme when the drift is not sublinear, that we overcome here by introducing a truncation-type function $\tronc$ (see \eqref{eq:formegeneralenoyautronc}).
\end{Rq}

\subsection{Applications} \label{sec:examples}

 We apply the general theorem to the two models introduced in the previous section.
 
\subsubsection{Euler schemes of McKean-Vlasov on bounded subsets and Approximation of QSDs of continuous-time dynamics}
We recall that the first model is an Euler scheme of a McKean Vlasov diffusion in a bounded open set $D$ of $\ER^d$. We provide  a set of assumptions which guarantee the convergence\footnote{By convergence, we mean abusively that tightness combined with the fact that weak limits are targets.} towards QSDs of these discrete dynamics and also show that when the stepsize $h$ goes to $0$, we also approximate QSDs of the McKean-Vlasov dynamics with \textit{hard killing}. Let us introduce our assumptions:\\

\begin{hmv}\hypertarget{hmv}{}
\item{}\label{condhmv} $x\mapsto b(x,\mu)$ and $x\mapsto \sigma(x,\mu)$ are Lipschitz continuous on $\MM$, uniformly in $\mu$.   For all $x\in D$, $b(x,\cdot)$ and $\sigma(x,\cdot)$ admit flat derivatives which are {bounded: 
\begin{align*}
    \sup_{(x,\mu)\in\MM\times\mathcal{P}(\MM)}\bigg\|\frac{\delta b(x,\cdot)}{\delta m}(\mu,\cdot)\bigg\|_{\infty}\vee\bigg\|\frac{\delta \sigma(x,\cdot)}{\delta m}(\mu,\cdot)\bigg\|_{\infty}<+\infty.
\end{align*}}

Finally, $x\mapsto\sigma(x,\mu)$ is uniformly elliptic, uniformly in $\mu$, $i.e.$ $\exists \lambda_0>0$ such that for all $(x,\mu)\in\MM\times {\cal P}(\MM)$, $\sigma\sigma^\star(x,\mu)\ge \lambda_0 {\rm Id}$ (in the sense of symmetric matrices).
\end{hmv}

\noindent Under  \ref{condhmv} and especially under the ellipticity assumption, we can consider the problem on $\bar{D}$ instead of $D$, $i.e.$ we are interested in the QSDs of $(K_{\mu,\partial}^{(h)})_{\mu\in {\cal P}(\bar{D})}$, where
\begin{equation}\label{eq:kmupartialh}
K_{\mu,\partial}^{(h)}f(x)=\ES[f(\Upsilon_x^h){\bf 1}_{\Upsilon_x^h\in \bar{D}}],
\end{equation}
with $\Upsilon_x^h$ defined in \Cref{rq: model1}. In order to emphasize the dependence in the step $h$, we also write $(\mu_n^h)$ instead of $(\mu_n)$ in the next proposition.
\begin{prop} \label{prop:model1} Assume \ref{condhmv} and $D$ has ${\cal C}^2$-boundary. Assume that $(\gamma_n)$ satisfies \eqref{hyp:pas}. Let $h>0$ and assume that $\zeta_h$ has a positive {${\cal C}^1$} density $g_h$.
\begin{itemize}
\item[(i)] Let $h>0$.  Then, $(\mu_n^h)$ is tight on ${\cal P}(\bar{D})$ and every weak limit of $(\mu_n^h)$ is a QSD $\mu^\star_h$ for $(K_{\mu,\partial}^{(h)})_{\mu \in {\cal P}(\bar{D})}$. Furthermore, $\mu^\star_h(\partial D)=0$. 
\item[(ii)] {Furthermore, if $(\zeta_t)_{t\ge0}$ is a Brownian motion and if
 \begin{align*}
    \sup_{(x,\mu)\in\MM\times{\cal P}(\MM)}\bigg[\frac{\delta b(x,\cdot)}{\delta m}(\mu,\cdot)\bigg]_{1}\vee\bigg[\frac{\delta\sigma(x,\cdot)}{\delta m}(\mu,\cdot)\bigg]_{1}<+\infty,
\end{align*}} such a sequence  $(\mu^\star_h)_{h>0}$ is tight on $D$ and every weak limit (when $h\rightarrow0$) is a QSD for the McKean-Vlasov diffusion $(Y_t)_{t\ge0}$, defined in \Cref{rq: model1} and killed at time $\tau:=\inf\{t\ge0, Y_t\in D^c\}$. 
\end{itemize}
\end{prop}
This result shows in particular that the algorithm may generate a way to simulate QSDs of continuous-time McKean Vlasov SDEs. In \Cref{sec:simu}, we provide several numerical illustrations.
\subsubsection{Dynamics killed with exit times of unbounded subsets}\label{subsec:unboundedsuperlinear}
In this section, we aim at considering unbounded examples. A prototypical example is a discretization scheme of  the following one-dimensional McKean-Vlasov equation killed when leaving $D=]0,+\infty[$:
\begin{equation}\label{eq:superlineardd}
 dY_t=(-\varphi(Y_t)+ \int F(Y_t-y)\mu_t(dy)) dt+ d\zeta_t,
 \end{equation}
where $F$ is supposed to be a bounded function and 
$$\lim_{x\rightarrow+\infty}\frac{\varphi(x)}{x}=+\infty.$$
We thus consider drift with superlinear growth: this will help to come from infinity but unfortunately, this will also generate unstability for discretized dynamics: when one considers the  genuine related Euler scheme given by the kernel 
$K_{\mu,\partial}^{(h)}$, the ``deterministic part''  will satisfy for each $h>0$:
$$ \lim_{x\rightarrow+\infty} x-h \varphi(x)+\int F(x-y)\mu(dy)=-\infty.$$
Among the consequences, this will imply for instance that 
$$\delta_\mu(x)\xrightarrow{x\rightarrow+\infty}1,$$
which in turn involves that the second assumption of the (tightness)  \Cref{prop:tightnessbis} will not be satisfied.
There are several ways to manage this discretization effect (for instance, one may consider adapted step sequences, see \emph{e.g.} \cite{lemaire}) but in this section where our aim is only to provide examples, we choose to introduce the following simple modification by considering:
\begin{equation}\label{eq:formegeneralenoyautronc}
 \tilde{K}_{\mu,\partial}^{(h)}f(x)=\ES[f(\tilde{\Upsilon}_x^h){\bf 1}_{\tilde{\Upsilon}_x^h\in D}]\quad\textnormal{with}\quad 
\tilde{\Upsilon}_x^h=\tronc(x+h b(x,\mu))+\sigma(x,\mu)\Delta \zeta_{h}.
\end{equation}
where, 
\begin{htronc}\hypertarget{htronc}{}
\item{}\label{condhtronc} 
The function $\tronc:\ER^d\mapsto\ER^d$ is ${\cal C}_1$ and 
$$\sup_{x,\mu\in \MM\times{\cal P}(\MM)} \{| \tronc(x+h b(x,\mu))|+\| D\tronc(x+h b(x,\mu))\|\}<+\infty.$$
\end{htronc}

\begin{Rq}  In order that $ \tilde{K}_{\mu,\partial}^{(h)}$ remains an approximation of the continuous dynamics, it is fundamental  to choose a function $\tronc$ which does very slightly modify the dynamics of the genuine Euler scheme. For instance, in the benchmark example \eqref{eq:superlineardd}, a possible choice for $\tronc$ is to define it by:
$$\tronc(x)=\begin{cases} x&\textnormal{if $x\ge -R$}\\
-R-1&\textnormal{if $x\ge -R-1$},
\end{cases}
$$
where $R$ is a positive number, and to take a smooth increasing interpolation on $(-R-1,-R)$. With such a function, one can check that $\mathbf{(H_\tronc)}$ holds true  for any $R>0$.
Such a modification becomes less and less restrictive when $h\rightarrow0$ so that a convergence to the continuous-time dynamics may be true (we do not prove such a property in this part).\\

The above strategy can be also considered in the multidimensional case but the choice of the function $\tronc$ may depend on the model and on the domain. We do not provide additional details.
\end{Rq}



We here consider the second model (Model 2) where the dynamics live in an unbounded subset $\MM$ of $\ER^d$. We have the following proposition:
\begin{prop} \label{prop:model2} Let $h>0$. Assume \ref{condhmv} and \ref{condhtronc}. Assume that
there exists $\alpha>0$ such that $\ES[|\zeta_h|^\alpha]<+\infty$ and that $\zeta_h$ has a ${\cal C}^1$ density $g_h$ w.r.t. the Lebesgue measure with $\|g_h\|_\infty+\|\nabla g_h\|_\infty<+\infty$. Also assume that the following condition holds: for every $M>0$ and $(s_1,s_2)\in(0,+\infty)^2$ with $s_1\le s_2$, there exist $c_1,c_2>0$,  and an integrable function $\mathfrak{p}_h:D\mapsto (0,+\infty)$ such that for every $m\in B(0,M)$ and $\sigma\in \mathbb{M}_{d,d}$ with $s_1\le |{\rm det}(\sigma)|\le s_2$,
\begin{equation}\label{eq:conddensminomajo}
\forall z\in D, \quad c_1 \mathfrak{p}_h(z)\le g_h\left(\sigma^{-1}(z-m)\right)\le c_2 \mathfrak{p}_h(z).
\end{equation}
Then, if $(\gamma_n)$ satisfies \eqref{hyp:pas} and $\sum \gamma_n^2<+\infty$, $(\mu_n)$ is almost surely tight in ${\cal P}(\bar{D})$ and every weak limit $\mu^\star_h$ is a QSD for $(\tilde{K}_{\mu,\partial}^{(h)})_{\mu \in {\cal P}(\bar{D})}$ which satisfies $\mu^\star_h(\partial D)=0$.
\end{prop}
\begin{Rq} 


$\rhd$ Let us comment Condition \eqref{eq:conddensminomajo}. This assumption is built in order to ensure the conditions of \Cref{lem:lowerupper} which in turn imply the (constraining) assumption \ref{condhquatre}. Condition  \eqref{eq:conddensminomajo} is for instance satisfied when $g_h$ is a positive continuous function such that
$$ g_h(x)\overset{|x|\rightarrow+\infty}{\sim}\frac{1}{|x|^\alpha}\quad \textnormal{with $\alpha>1$.}$$
This allows to consider a large family of L\'evy processes (such as symmetric stable processes) but unfortunately, it does not allow to consider the Brownian motion. In this case, it seems difficult to obtain the ``lower/upper-bound'' of the kernel. Nevertheless, we can only lower-bound and this implies that one preserves the first part of \Cref{theo:discret}.


\begin{Rq} Let us finally comment about the mean-reverting SDEs with sublinear drift, such as Ornstein-Uhlenbeck processes. We do not consider such dynamics in the non-compact case since Assumptions \ref{condhtrois} and \ref{condhquatre} do not hold in this case. Nevertheless, we show in \Cref{lem:hvdeux} that Assumption \ref{condhvdeux} may be true in this setting. Thus, one can obtain tightness of $(\mu_n)$. It is interesting to remark that such a property is new with respect to \cite{BCP} (where only the compact case was studied), even in the case of linear dynamics (by linear, we mean without interaction). Owing to an adaptation of \cite[Theorem 6.4]{BCP} (which does not require \ref{condhtrois} and \ref{condhquatre}), it may be shown that, in the linear setting, every limit point of $(\mu_n)$ is a QSD.
\end{Rq}

\end{Rq}
%
%
%
%


\subsection{Simulations}\label{sec:simu} We provide several numerical illustrations related to the Euler scheme of the following example
\begin{align*}
    d{\xi}_{t}=\gamma \mathbb{E}[{\xi}_{t}|\tau_D>t]dt+dW_{t},
\end{align*}
with $D=]-1,1[$. This example comes from \cite{tough_nolen} (Example 1.6) where the authors are able to give explicit expressions of the related QSDs, taking advantage of the fact that 
$b(x,\mu)=\int x\mu(dx)$, so that $x\mapsto b(x,\mu)$ is constant. In this simple example, it can be shown that the QSDs have densities given by
$$ \pi_b\propto e^{b x}\cos(\frac{\pi}{2} x),$$
where  $b$ is a fixed point of the map 
$$b\mapsto \tanh(\gamma b)-\frac{8\gamma b}{4\gamma^2 b^2+\pi^2}.$$
It can be checked that
\begin{itemize}
\item{} If $\gamma\le \frac{\pi^2}{\pi^2+8}$, 
$b=0$ is the only fixed point so that the unique QSD is $\pi_0=\frac{4}{\pi}\cos(\frac{\pi}{2}x){\bf 1}_{[-1,1]}(x)$.
\item{}  If $\gamma> \frac{\pi^2}{\pi^2+8}$, there are three fixed points $b_-,0,b_+$ with $b_{-}=-b_{+}$ which lead to three QSDs denotes $\pi_{-}$, $\pi_0$ and $\pi_{+}$.
\end{itemize}
In \Cref{fig:sim1}, we compare true (continuous line) and empirical densities (dotted lines). By empirical density, we mean here the convolution of $\mu_n$ with a Gaussian kernel. 
For  $\gamma=0.5$, the convergence holds to the unique QSD whereas for  $\gamma=4$, it is interesting to remark that the three QSDs are possible limit points of the sequence. We chose the initial points $-0.5$, $0$ and $0.5$ in order to increase the probability  of approximating $\pi_{-}$, $\pi_0$ and $\pi_+$ but, of course, starting from $0.5$ for instance, there is a positive probability that it falls in the basin of attraction of $\pi_{-}$ or $\pi_0$. Nevertheless, since the three equilibriums are isolated, it is probably true that the algorithm converges to one of the three QSDs. 
            \begin{figure}[htbp]
           \begin{minipage}{0.45\linewidth}
              \includegraphics[width=\linewidth]{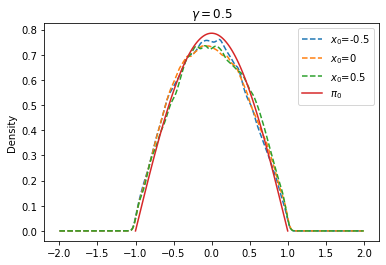}
                       \end{minipage}
      \begin{minipage}{0.45\linewidth}
        \includegraphics[width=\linewidth]{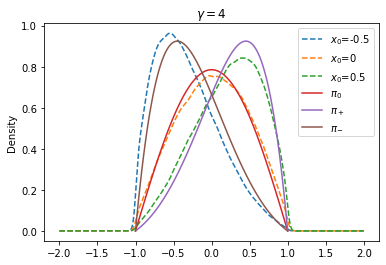}
                        \end{minipage}
                        \caption{Left. $(\gamma,h,n)=(0.5,0.01,10^5)$. Right. $(\gamma,h,n)=(4,0.01,5.10^5)$. }\label{fig:sim1}
 \end{figure}
 To finish these short numerical illustrations, we consider the case $\gamma=1$ in \Cref{fig:sim2}. In this case, there are three QSDs but our simulations suggest that only $\pi_0$ is attractive.
  \begin{figure}[htbp]
  \includegraphics[width=6cm]{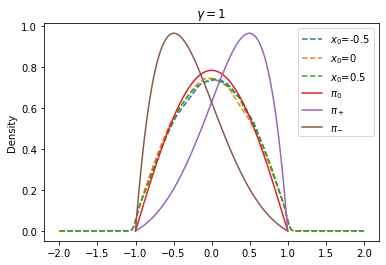}
                        \caption{$(\gamma,h,n)=(1,0.01,10^5)$. }\label{fig:sim2} \end{figure}

\section{Preliminaries, Poisson Equation and limiting dynamics}\label{sec:proof1}
\label{sect:Prelim}
We begin by giving properties related to $\Pi_\mu$ and ${\cal A}_\mu$. Then, we introduce and study the Poisson equation related to the kernel $\kmu$. This study is one of the cornerstones of \cref{prop:asymptoticpseudotrajectory} which will allow us to characterize the limiting dynamics of $(\tilde{\mu}^{(n)})_{n\ge1}$.\\

\noindent \textbf{ Explicit form for $\Pi_\mu$.} Recall that  $ \Aa_\mu$ is a kernel defined by
$$ \Aa_\mu(x,.)=\sum_{n\ge0} K_{\mu,\partial}^n(x,.).$$
\begin{lem}  \label{lem:basics}
Assume \ref{condhzero} and \ref{condhun}. Then,\\

\noindent (i)  $\Aa_\mu$ is well-defined for every $\mu\in{\cal P}(\MM)$ and $\|\Aa_\mu{\bf 1}\|_\infty<+\infty$.\\

\noindent (ii) The probability $\Pi_\mu$ defined by
\begin{equation}\label{explicitpimu}
\Pi_\mu=\frac{\mu \Aa_\mu}{(\mu \Aa_\mu)(\bf 1)}
\end{equation}
is the unique invariant distribution for $\kmu$.\\

\noindent (iii) $\mu^\star$ is a QSD for $(K_{\mu,\partial})_{\mu\in{\cal P}(\MM)}$ if and only if 
$\mu^\star=\Pi_{\mu^\star}$. Existence holds for the QSD under the assumptions of \Cref{lem:characQSD}$(ii)$.
\end{lem}
\begin{proof}[Proof of  \Cref{lem:basics}] $(i)$ By \ref{condhzero} and an induction,
 \begin{equation}\label{controlesurvie}
 \PE_x(\tau_{\mu}>n\ell)=K_{\mu,\partial}^{n\ell}{\bf 1}(x)\le \rho^n.
 \end{equation}
 Thus, 
 $$\Aa_\mu{\bf 1}(x)=\sum_{n\ge0}\PE_x(\tau_{\mu}>n)=\ES_x[\tau]\le \frac{\ell}{1-\rho}.$$
 
 \noindent $(ii)$ For any bounded continuous function $f$,
\begin{align*}
\mu \Aa_\mu \kmu(f)&=\mu\left(\sum_{n\ge0} (K_{\mu,\partial}^{n+1} f+K_{\mu,\partial}^n\delta_\mu \mu(f))\right)=\sum_{n\ge0}\mu K_{\mu,\partial}^{n+1}f+\mu(f)\mu(\sum_{n\ge0} K_{\mu,\partial}^n(\delta_\mu))\\
&=\mu \Aa_\mu f
\end{align*}
since $\delta_\mu={\bf 1}-K_{\mu,\partial}{\bf 1}$, which implies that
$$\sum_{n\ge0} K_{\mu,\partial}^n(\delta_\mu))=\sum_{n\ge0} K_{\mu,\partial}^n{\bf 1}
-\sum_{n\ge0} K_{\mu,\partial}^{n+1}{\bf 1}={\bf 1}.$$
To prove uniqueness, let us set 
\begin{equation}\label{eq:poisssss}
P_\mu=e^{-1}\sum_{k\ge 0} \frac{\kmu^k}{k!}.
\end{equation}
One easily checks that if  $\Pi_\mu$ is an invariant distribution for $\kmu$, $\Pi_\mu$ is also invariant for $P_\mu$ and hence, it is enough to prove uniqueness for $P_\mu$.
By \eqref{supkmufini} and the very definition of $\kmu$, we have for every $k\ge1$,
$$\kmu^k\ge (K_{\mu,\partial}^{k-1}\delta) \mu=(K_{\mu,\partial}^{k-1}{\bf 1}-K_{\mu,\partial}^{k}{\bf 1})\mu,$$
so that there exists $c>0$ such that
$$ P_\mu\ge c\sum_{k=1}^{\ell-1} {\kmu^k}\ge c\left(\sum_{k=1}^{\ell-1}K_{\mu,\partial}^{k-1}{\bf 1}-K_{\mu,\partial}^{k}{\bf 1}\right)\mu=c(1-K_{\mu,\partial}^{\ell}{\bf 1})\mu\ge \varepsilon \mu$$
with $\varepsilon=c(1-\rho)>0.$
This property classically implies uniqueness of the invariant distribution of $P_\mu$ (by a coupling argument).\\

\noindent $(iii)$ The definition of $\kmu$ implies that a probability $\Pi_\mu$ is invariant if and only if:
$$\Pi_\mu K_{\mu,\partial}=(\Pi_\mu K_{\mu,\partial} {\bf 1})\mu.$$
Thus, $\mu=\Pi_\mu$ if and only if 
\begin{equation}\label{eq:fixedpoint2}
\mu K_{\mu,\partial}=(\mu K_{\mu,\partial} {\bf 1})\mu.
\end{equation}
Now, if $\mu$ is a QSD, then \eqref{eq:fixedpoint2} holds true since it corresponds to the definition at time $1$: for all Borel set $A$ of $\MM$,
$$\mu(A)(\mu K_{\mu,\partial} {\bf 1})=\mu(A)\PE_\mu(\tau>1)=\PE_\mu(Y_1\in A, \tau>1)=(\mu K_{\mu,\partial})(A).$$
Conversely, if \eqref{eq:fixedpoint2} holds, let us first remark that $\mu$ is trivially a QSD if $\mu K_{\mu,\partial}{\bf 1}=0$. Otherwise,  one shows by induction that \eqref{def:nbyn} holds true. The initialization follows from the above equality. If the property is true until time $n$, then 
$\PE_\mu(\tau>n)=(\mu K_{\mu,\partial}{\bf 1})^n$ and ${\cal L}(Y_n|\tau>n)=\mu$. Noting that for all Borel set $A$ of $\MM$, $\{Y_{n+1}\in A\}\subset \{\tau>n+1\}$,
we then  deduce from the Markov property that
\begin{equation}\label{eq:preuveQSD}
\begin{split}
\PE(Y_{n+1}\in A, \tau>n+1)&=\PE(Y_{n+1}\in A,\tau>n)=\ES_\mu[ K_{\mu,\partial}(x,A) {\bf 1}_{\tau>n}]\\
&=\mu K_{\mu,\partial}(A)\PE_\mu(\tau>n),
\end{split}
\end{equation}
where in the last equality, we used the recurrence hypothesis applied to $f$ defined by $f(x)=K_{\mu,\partial}(x,A)$.
By \eqref{eq:fixedpoint2}, we deduce that
\begin{equation*}
\PE(Y_{n+1}\in A, \tau>n+1)=\mu(A)(\mu K_{\mu,\partial} {\bf 1})\PE_\mu(\tau>n)=\mu(A)\PE_\mu(\tau>n+1),
\end{equation*}
where in the last inequality, we used \eqref{eq:preuveQSD} with $A=\MM$. The equivalence is proved.\\

\noindent Let us prove existence of the QSD. First, note that if there exists $\mu\in{\cal P}(\MM)$ such that
$\mu K_{\mu,\partial}{\bf 1}=0$, then $\mu$ is a trivial QSD (since $\PE_\mu(\tau>1)=0$ in this case). Thus, we can suppose that 
$\mu K_{\mu,\partial}{\bf 1}>0$ for every $\mu\in{\cal P}(\MM)$. By \ref{condhun}, this involves that the map
$$\Upsilon : \mu\mapsto \frac{\mu K_{\mu,\partial}}{\mu K_{\mu,\partial}{\bf 1}}$$
is continuous on ${\cal P}(\MM)$. From the first part of the proof of $(iii)$ (see \eqref{eq:fixedpoint2}), a fixed point of $\Upsilon$ is  a QSD. Thus, it follows from Kakutani fixed point theorem
that existence will hold for the QSD as soon as there exist $\varepsilon$ and $M$ such that $\Upsilon$ is stable on the  convex compact\footnote{The set ${\cal V}_{M,\varepsilon}$ is clearly convex and the inf-compactness of $V_\varepsilon$ implies that ${\cal V}_{M,\varepsilon}$ is a closed tight subset of ${\cal P}(\MM)$ and thus a compact subset of ${\cal P}(\MM)$. } subspace ${\cal V}_{M,\varepsilon}$ of ${\cal P}(\MM)$ defined (with the notations of \ref{condhvdeux}) by
$${\cal V}_{M,\varepsilon}=\{\mu\in{\cal P}(\MM), \mu(V_\varepsilon)\le M\}.$$
If  $\alpha:=\inf_{\mu \in{\cal P}(\MM)} \mu K_{\mu,\partial}{\bf 1}>0$, and \ref{condhvdeux} holds, then for all $\varepsilon>0$,
$$ \Upsilon(\mu)(V_\varepsilon)\le \frac{\varepsilon}{\alpha} \mu(V_\varepsilon)+\frac{\beta_{\varepsilon}}{\alpha}.$$
Taking $\varepsilon=\frac{\alpha}{2}$, we get
$$ \forall \mu \in{\cal P}(\MM),\quad \Upsilon(\mu)(V_\varepsilon)\le \frac{1}{2} \mu(V_\frac{\alpha}{2})+\frac{\beta_{\alpha/2}}{\alpha}.$$
Thus, setting $M= \frac{2\beta_{\alpha/2}}{\alpha}$, we get 
$$\Upsilon({\cal V}_{M,\frac{\alpha}{2}})\subset {\cal V}_{M,\frac{\alpha}{2}}.$$
This concludes the proof.
\end{proof}
\begin{lem}\label{lem:pimulip} Assume \ref{condhzero}, \ref{condhun} and \ref{condhdeux}. Then, $\mu\mapsto \Aa_\mu$ is Lipschitz for the TV-distance, uniformly in $x\in\MM$, and $\mu\mapsto \Pi_\mu$ is Lipschitz for the TV-distance. 
\end{lem}
\begin{proof}
Recall that $\Aa_\mu=\sum_{n\ge0} K_{\mu,\partial}^n$. We use the following decomposition:
$$  K_{\mu,\partial}^n- K_{\nu,\partial}^n=\sum_{k=1}^n K_{\mu,\partial}^{n-k}(K_{\mu,\partial}-K_{\nu,\partial}) K_{\nu,\partial}^{k-1}.$$
By \ref{condhzero}, $\sup_{\mu} \|K_{\mu,\partial}^\ell {\bf 1}\|_{\infty}<1,$
so that using that $n\mapsto K_{\mu,\partial}^n {\bf 1}$ is non-increasing, we deduce that there exists a constant $C$ and a $\rho<1$ such that 
$$\sup_{\mu} \|K_{\mu,\partial}^n {\bf 1}\|_{\infty}\le C \rho^n.$$
Thus, for every bounded measurable function $f$ on $\MM$, for every $x\in\MM$,
\begin{align*}
|K_{\mu,\partial}^{n-k}(K_{\mu,\partial}-K_{\nu,\partial}) K_{\nu,\partial}^{k-1}f(x,.)&\le C\rho^{n-k} \|(K_{\mu,\partial}-K_{\nu,\partial})K_{\nu,\partial}^{k-1}f\|_{\infty}\\
&\le C\rho^{n-k} \|\mu-\nu\|_{TV} \|K_{\nu,\partial}^{k-1} f\|_{\infty}\le C \|f\|_\infty \rho^n \|\mu-\nu\|_{TV},
\end{align*}
where $C$ may have changed from line to line and where we used \ref{condhdeux}.
This means that
\begin{align*}
\sup_{x\in\MM} \|K_{\mu,\partial}^{n-k}(K_{\mu,\partial}-K_{\nu,\partial}) K_{\nu,\partial}^{k-1}(x,.)\|_{TV}
\le C \rho^n \|\mu-\nu\|_{TV},
\end{align*}
and it follows that 
$$\sup_{x\in\MM} \|\Aa_\mu(x,.)-\Aa_\nu(x,.)\|_{TV}\le \sum n\rho^n<+\infty.$$
This easily implies that $\mu\mapsto \mu\Aa_\mu$ and thus $\mu\mapsto \Aa_\mu{\bf 1}$ are Lipschitz for the TV-norm. Using that 
$\mu \Aa_\mu{\bf 1}\ge 1$, one then deduces that $\mu\mapsto \Pi_\mu$ is also Lipschitz.
\end{proof}

\noindent \textbf{Poisson Equation related to $\kmu$.}
Let $(\kmu)_{\mu\in{\cal P}({\MM})})$ denote a family of transition kernels satisfying the following assumptions:
\begin{pzero}\hypertarget{pzero}{}
\item{}\label{condpzero} For every $\mu \in{\cal P}(\MM)$, $\kmu$ admits a unique invariant distribution ${\Pi}_\mu$ and $\mu\mapsto \Pi_\mu$ is Lipschitz continuous on ${\cal P}(\MM)$ for the TV-distance.
\end{pzero}
\begin{pun}\hypertarget{pun}{}
\item{}\label{condpun} There exist $\ell \in \mathbb{N}$, $\ell \ge1$, $\varepsilon>0$ and $\Psi\in{\cal P}({\MM})$ such that for every $\mu\in{\cal P}(\MM)$,
$$ \kmu^\ell(.,dy)\ge \varepsilon \Psi(dy).$$
\end{pun}
\begin{pdeux}\hypertarget{pdeux}{}
\item{}\label{condpdeux} $\mu \mapsto \kmu(x,.)$ is Lipschitz continuous for the ${\rm TV}$-distance on ${\cal P}({\MM})$, uniformly in $x$; namely, there exists $c_2>0$ such that for every $x\in{\cal P}(\MM)$ and every $\mu$ and $\nu \in{\cal P}(\MM)$,
$$\|\kmu(x,.)-\knu(x,.)\|_{TV}\le c_2\|\mu-\nu\|_{TV}.$$
\end{pdeux}
For a fixed $\mu$ and a given function $f:{\MM}\rightarrow\ER$, let us consider the Poisson equation
\begin{equation}\label{Poissonequa}
 f-\Pi_\mu f= ({\rm Id}-\kmu) g.
 \end{equation}
We have the following result:
\begin{lem}[Poisson equation]\label{prop:poisequa}
Assume \ref{condpzero}, \ref{condpun} and \ref{condpdeux}. For any $\mu\in{\cal P}({\MM})$, for any measurable function  $f:{\MM}\rightarrow\ER$, the Poisson equation \eqref{Poissonequa} admits a unique solution denoted $Q_\mu f$ and
defined by
\begin{equation}
\label{eq:def-Q}
 Q_\mu f(x)=\sum_{n=0}^{+\infty} \left(\kmu^n f(x) -\Pi_\mu (f)\right).
\end{equation}
Furthermore, there exists a constant $C$ such that,
\begin{itemize}
\item[(a)]  For every $\mu,\alpha \in {\cal P}(\MM)$,  
$$| \alpha Q_\mu f | \leq C \Vert f \Vert_\infty \Vert \alpha - \Pi_\mu \Vert_{VT}.$$
\item[(b)] For every $\mu \in{\cal P}(\MM)$, 
$$\Vert Q_\mu f - Q_\nu f \Vert_\infty \leq C \Vert f \Vert_\infty  \Vert \mu- \nu \Vert_{TV}.$$
\end{itemize}
\end{lem}
\begin{Rq}\label{rem:poissonization}  $\rhd$ This result will be applied to $\kmu$ defined by \eqref{eq:kmudef} but it applies to every family of kernels satisfying \ref{condpzero}, \ref{condpun} and \ref{condpdeux}.\\

\noindent $\rhd$ In the proof of this result, one can check that the fact that $\Psi$ does not depend on $\mu$ only plays a role in \eqref{eq:deppsi} and that the result is still true if Assumption \ref{condpun} is satisfied with a probability $\Psi_\mu$ where $\mu\mapsto \Psi_\mu$ is Lipschitz continuous for the TV-distance. For instance, 
if 
\begin{equation}\label{eq:additionallowerbound}
\varepsilon= \inf_{(x,\mu)\in{\cal P}(\MM)}\delta_\mu(x)>0,
\end{equation} 
the rebirth part of the dynamics implies that
$$\kmu\ge\varepsilon \mu$$
In this case, the result is still true. Such a variation may allow to remove \ref{condhtrois} from the assumptions of \Cref{theo:discret}$(i)$. One way to definitely remove \ref{condhtrois} from \Cref{theo:discret}$(i)$ (without assuming  the additional assumption \eqref{eq:additionallowerbound} for instance) is to use a Poissonization argument like in \eqref{eq:poisssss} which allows to lower-bound the (Poissonized) dynamics by $\varepsilon \mu$ with very slight assumptions. For details in this direction, we refer to \cite[Lemma 6.3]{BCP}.
\end{Rq}

\begin{proof}
\textit{(a)} We prove $(a)$ and the well-definition of $Q_\mu f$ in the same time. By \ref{condpun}, 
\begin{equation*}
\forall x\in {\MM}, \quad \kmu^\ell(x,.)=\varepsilon \Psi(.)+(1-\varepsilon) R_\mu(x,.)
\end{equation*}
where 
$$R_\mu=\frac{\kmu^\ell-\varepsilon\Psi}{1-\varepsilon},$$
is a transition kernel. Using that for a probability $\alpha \in{\cal P}(\MM)$, $\alpha\Psi=\Psi$, it follows that for any $\alpha,\beta\in{\cal P}(\MM)$.
$$(\alpha-\beta) \kmu^\ell=(1-\varepsilon)(\alpha-\beta) R_\mu.$$
Then, an induction leads to
\begin{equation}\label{eq:RMUn}
(\alpha-\beta) \kmu^{n\ell}=(1-\varepsilon)^n (\alpha-\beta) R_\mu^n.
\end{equation}
and to
$$\|\alpha \kmu^{n\ell}-\beta \kmu^{n\ell}\|_{TV}\le (1-\varepsilon)^n\|\alpha-\beta\|_{TV},$$
for any $\alpha,\beta\in{\cal P}(\MM)$. In particular, 
$$\|\alpha \kmu^{n\ell}-\Pi_{\mu}\|_{TV}\le (1-\varepsilon)^n\|\alpha-\Pi_{\mu}\|_{TV}, $$
and for every $r\in\llbracket 0,\ell-1\rrbracket$,
$$\|\alpha \kmu^{n\ell+r}-\Pi_{\mu}\|_{TV}\le (1-\varepsilon)^n\|\alpha \kmu^r-\Pi_{\mu} \kmu^r\|_{TV} \le (1-\varepsilon)^n\|\alpha-\Pi_{\mu}\|_{TV}. $$
In the above inequality, we used that if $K$ is a transition kernel $\|\alpha K-\beta K\|_{TV}\le \|\alpha-\beta\|_{TV}$. These inequalities easily lead to $(a)$ and to the well-definition of $Q_\mu$. Checking that $({\rm Id}-K_\mu)Q_\mu f=Q_\mu ({\rm Id}-K_\mu) f$ classically implies that $Q_\mu f$ is the unique solution of the Poisson equation.\\

\noindent $(b)$ Let us  prove that $\mu \to Q_\mu f(x)$ is Lipschitz continuous. From the definition of $Q_\mu$ and from \eqref{eq:RMUn}, we have
\begin{align*}
 Q_\mu f(x) - Q_\nu f(x)
&=\sum_{n=0}^{+\infty} \left( (\kmu^n f(x) -\Pi_\mu \kmu^n (f)) - \left(\knu^n f(x) -\Pi_\nu \knu^n (f)\right)\right) \\
&= \sum_{n=0}^{+\infty} \sum_{j=0}^{\ell-1} \left( (\kmu^{n\ell+j} f(x) -\Pi_\mu \kmu^{n\ell+j} (f)) - (\knu^{n\ell+j} f(x) -\Pi_\nu \knu^{n\ell +j} (f)) \right) \\
&= \sum_{n=0}^{+\infty} (1- \varepsilon)^n \sum_{j=0}^{\ell-1} \left( (\delta_x - \Pi_\mu) R^n_\mu \kmu^{j} f - (\delta_x - \Pi_\nu) R^n_\nu \knu^{j} f  \right).
\end{align*}
Now, assume without loss of generality that $\|f\|_\infty\le 1$. In this case, one can check that
for every $n\ge0$, $j\in\llbracket 0,\ell-1\rrbracket$, \begin{align*}
\Big| (\delta_x - &\Pi_\mu) R^n_\mu \kmu^{j} f - (\delta_x - \Pi_\nu) R^n_\nu \knu^{j} f  \Big|\le 2  \sup_{\alpha}\|\alpha R_\mu^n \kmu^j -\alpha R_\nu^n\knu^j  \|_{TV}  +\|\Pi_\mu-\Pi_\nu\|_{TV}\\
&\le  2\sup_{\alpha}\|\alpha R_\mu^n-\alpha R_\nu^n\|_{TV}+2\sup_{\alpha} \|\alpha\kmu^j-\alpha\knu^j\|_{TV}+\|\Pi_\mu-\Pi_\nu\|_{TV}.
\end{align*}
By \ref{condpzero} and \Cref{lem:lip-prel}, one deduces that a constant $C$ exists such that
$$\Big| (\delta_x - \Pi_\mu) R^n_\mu \kmu^{j} f - (\delta_x - \Pi_\nu) R^n_\nu \knu^{j} f  \Big|\le C(1+n)\|\mu-\nu\|_{TV}.$$
Thus, 
$$| Q_\mu f(x) - Q_\nu f(x)|\lesssim \|\mu-\nu\|_{TV} \sum_{n=0}^{+\infty}(1- \varepsilon)^n (1+n)\lesssim\|\mu-\nu\|_{TV}.$$
This concludes the proof.

\end{proof}

\begin{lem}\label{lem:lip-prel}
Assume \ref{condpzero}, \ref{condpun} and \ref{condpdeux}. Then,
\begin{itemize}
\item[(i)] For all $n\in \mathbb{N}$,
$$\sup_{\alpha \in {\cal P}(\MM)} \Vert \alpha \kmu^{n} - \alpha \knu^{n} \Vert_{TV} \le c_2 n\|\mu-\nu\|_{TV}.$$
\item[(ii)] Setting
$R_\mu=(1-\varepsilon)^{-1}(\kmu^\ell-\varepsilon\Psi)$, we have
$$\sup_{\alpha\in{\cal P}(\MM)} \Vert \alpha R_\mu^{n}  - \alpha  R_\nu^n   \Vert_{TV}\le \frac{c_2 n\ell}{1-\varepsilon}\|\mu-\nu\|_{TV}.$$
\end{itemize}
\end{lem}
The proof of \cref{lem:lip-prel} is postponed to \cref{sec:proofsappendix}. In the proposition below, we obtain a result which will lead to the identification of the limit. Note that this result does not directly depend on the dynamics of our algorithm but only on the generic assumptions \ref{condpzero}, \ref{condpun} and \ref{condpdeux}. For the sake of simplicity, we use the notations of the algorithm but write \textit{ad hoc} assumptions (which are satisfied in this setting) under which the conclusion is true.





\begin{prop}\label{prop:asymptoticpseudotrajectory}
Assume \ref{condpzero}, \ref{condpun} and \ref{condpdeux}. Let $(X_n,\mu_n)_{n\ge1}$ denote a random sequence with values in $\MM\times{\cal P}(\MM)$ such that 
$\mu_n$  is $\sigma(X_k,k\le {n-1})$-measurable and  ${\cal L}(X_{n+1}|\sigma(X_k, k\le n))=\mathbb{K}_{\mu_n}(X_n,.)$. Then, if $(\gamma_n)_{n\ge1}$ denotes a non-increasing sequence of positive numbers satisfying \eqref{hyp:pas} and $\|\mu_{n+1}-\mu_n\|_{TV}\le C\gamma_{n+1}$ $a.s.$ where $C$ does not depend on $n$, 
then for any $T>0$, for any bounded measurable function $f:\MM\mapsto \ER$
$$\lim_{n\rightarrow+\infty}\sup_{m\in\llbracket n,N(n,T)\rrbracket }\left|\sum_{k=n}^{m} \gamma_k\left(f(\Zee_k)-\Pi_{\mu_k}(f)\right)\right|=0,$$
where 
$$N(n,T):=\inf\{m\ge n, \sum_{k=n}^{m+1}\gamma_k>T\}.$$
\end{prop}
\begin{proof} Without loss of generality, we assume that $\|f\|_\infty\le 1$.
By \eqref{Poissonequa},
\begin{align*}
f(\Zee_{n+1})-\Pi_{\mu_n} (f)&=Q_{\mu_n} f(\Zee_{n+1})-\mathbb{K}_{\mu_n} Q_{\mu_n} f(\Zee_{n+1})\\
&=\gamma_{n+1}\Delta M_{n+1}(f)+\Delta R_{n+1}(f)+\gamma_{n+1}\Delta D_{n+1}(f),
\end{align*}
with,
\begin{align*}
&\Delta M_{n+1}(f)= Q_{\mu_n} f(\Zee_{n+1})-\mathbb{K}_{\mu_n}Q_{\mu_n} f(\Zee_{n})\\
&\Delta R_{n+1}(f)=(\gamma_{n+1}-\gamma_n)\mathbb{K}_{\mu_n}Q_{\mu_n} f(\Zee_{n})+ \left(\gamma_n \mathbb{K}_{\mu_n}Q_{\mu_n} f(\Zee_{n})-\gamma_{n+1} \mathbb{K}_{\mu_{n+1}}Q_{\mu_{n+1}} f(\Zee_{n+1})\right)\\
&\Delta D_{n+1}(f)= \left(\mathbb{K}_{\mu_{n+1}}Q_{\mu_{n+1}} f(\Zee_{n+1})-\mathbb{K}_{\mu_n} Q_{\mu_n} f(\Zee_{n+1})\right).
\end{align*}
First, let us focus on $\Delta R_{n+1}$.  Using for the first part that $(\gamma_n)_{n\ge0}$ is non-increasing and that $(x,\mu)\mapsto \mathbb{K}_\mu Q_\mu f(x)$ is (uniformly) bounded (by \Cref{prop:poisequa}), and a telescoping argument for the second part yields the existence of a constant $C$ such that for any positive integer $m$:
 $$|\sum_{k=n}^{m} \Delta R_{k}(f)|\le C\gamma_n.$$
 Second $(\Delta M_n)$ is a sequence of $({\cal F}_n)$-martingale increments.
 From \Cref{prop:poisequa}, $\Delta M_n(f)$ is bounded. As a consequence, by an adaptation of \cite[Proposition 4.4]{B99} (based on exponential martingales), 
 one is able to obtain that for every  $T>0$, some positive constants $C_1$ and $C_2$ exist such that
 $$\PE\left( \max_{m\in\llbracket n,N(n,T)\rrbracket }\left|\sum_{k=n}^{m}\gamma_k \Delta M_k(f)\right|\ge \varepsilon\right)\le C_1\exp\left(-C_2\gamma_{n}^{-1}\varepsilon^2\right),$$
and thus, from the Borel-Cantelli lemma and the assumption $\lim_{n\rightarrow+\infty}\gamma_n\log(n)=0$, 
that for every $T>0$
 $$\limsup_{n\rightarrow+\infty} \max_{m\in\llbracket n,N(n,T)\rrbracket }\left|\sum_{k=n}^{m}\gamma_k \Delta M_k(f)\right| =0\quad a.s.$$
Finally, for the last term, one uses that $\mu\mapsto \mathbb{K}_\mu f$ and $\mu\mapsto Q_\mu f$  are Lipschitz continuous.
More precisely, using Lemma \ref{prop:poisequa} \textit{(i)}, \textit{(iii)} and the last inequality of Lemma \ref{lem:lip-prel}, we see that there exists $C>0$ such that
\begin{align*}
\Big|&\mathbb{K}_{\mu_{n+1}} Q_{\mu_{n+1}} f(\Zee_{n+1})-\mathbb{K}_{\mu_n} Q_{\mu_n} f(\Zee_{n+1}) \Big| \\
&\leq  \left|\mathbb{K}_{\mu_{n+1}}Q_{\mu_{n+1}} f(\Zee_{n+1})-\mathbb{K}_{\mu_{n+1}} Q_{\mu_n} f(\Zee_{n+1}) \right|+ \left|\mathbb{K}_{\mu_{n+1}}Q_{\mu_{n}} f(\Zee_{n+1})-\mathbb{K}_{\mu_n} Q_{\mu_n} f(\Zee_{n+1})\right| \\
  &\leq  \Vert Q_{\mu_{n+1}} f- Q_{\mu_n} f \|_\infty + \left\|(\mathbb{K}_{\mu_{n+1}}-\mathbb{K}_{\mu_n}) (Q_{\mu_n} f)\right\|_\infty \\
&\leq C \Vert f \Vert_\infty \Vert \mu_{n+1} - \mu_n \Vert_{TV} \leq C \gamma_{n+1},
\end{align*}
This ends the proof (since $\sum_{k=n+1}^{N(n,T)} |\Delta D_{k}(f)|\le C\sum_{k=n+1}^{N(n,T)}\gamma_k\le  C T$).

\end{proof}

\section{Long-time behavior of the limiting dynamics}\label{sec:proof2}
The aim of this section is to provide long-time properties of the limiting ODE
\begin{equation}\label{eq:nutdyn}
\dot{\nu}=-\nu+\Pi_\nu.
\end{equation}

To this end, we use  that this equation can be connected to an {``almost'' linear differential} equation through the following property (see \cref{rq:almostlinear} for details). 
 
\begin{lem}\label{lem:changevariable} Assume \ref{condhzero}, \ref{condhun} and \ref{condhdeux}. 
\begin{itemize} 

\item[(i)] For any $\mu_0\in{\cal P}(\MM)$, the differential equation
\begin{equation}\label{eq:linearizedmut}
\dot{\mu}_t=\mu_t {\cal A}_{\mu_t}-(\mu_t {\cal A}_{\mu_t}{\bf 1})\mu_t.
\end{equation}
admits a unique solution in ${\cal P}(\MM)$.
Furthermore, 
\begin{equation}\label{def:mut45}
\forall t\ge0,\quad \mu_t=\frac{\nu_0\varphi_t}{\nu_0\varphi_t{\bf 1}},
\end{equation}
where $(\varphi_t)_{t\ge0}$ is the unique solution on $({\cal B}(\MM),\|.\|_\infty)$, space of bounded linear operators on ${\cal C}(\MM,\ER_+)$ starting from ${\rm Id}$ to 
\begin{equation}\label{eq:almostlinear}
\dot{\varphi}=\varphi_{t} {\cal A}_{\mu_t}, \quad t\ge0.
\end{equation}
\item[(ii)] For every $\nu_0\in{\cal P}(\MM)$, the differential equation \eqref{eq:nutdyn} admits a unique solution $(\nu_t)_{t\ge0}$ in ${\cal P}(\MM)$, starting from $\nu_0$. Furthermore, 
if $(\mu_t)_{t\ge0}$ is the solution of \eqref{eq:linearizedmut}, then $\nu_t:=\mu_{\tau(t)}$ where $\tau$ is a ${\cal C}^1$-diffeomorphism from $\ER_+$ to $\ER_+$ being the unique solution of 
\begin{equation}\label{eq:tauu}
 \tau'(t)=\frac{1}{\mu_{\tau(t)} {\cal A}_{\mu_{\tau(t)}}{\bf 1}}\quad \textnormal{with $\tau(0)=0$}.
 \end{equation}

\end{itemize}
\end{lem}
\begin{Rq}\label{rq:almostlinear}  For a given $(\mu_t)_{t\ge0}$, the equation \eqref{eq:almostlinear} is an inhomogeneous linear equation. This will allow us to provide a series expansion of its solution (see \eqref{eq:semiexplicit}) and thus to $(\mu_t)_{t\ge0}$ by \eqref{def:mut45} (which depends on $(\mu_t)_{t\ge0}$ itself !). This is why we used the terminology ``almost linear''.
 \end{Rq}
\begin{proof} 
\noindent $(i)$ 
To prove existence and uniqueness of the solution, we apply the Cauchy-Lipschitz theorem on ${\cal M}_s(\MM)$ which denotes the space of finite signed measures on $\MM$ equipped with the total variation norm $\| .\|_{TV}$. It is well-known that ${\cal M}_s(\MM)$ is a Banach space. Since the operator ${\cal A}_\mu$ is not defined on ${\cal M}_s(\MM)\backslash {\cal P}(\MM)$, we extend the ODE on  ${\cal M}_s(\MM)$  by considering
\begin{equation}\label{eq:extendedode}
\dot{\mu}=\mu {\cal A}_{p(\mu)}-(\mu {\cal A}_{p(\mu)}{\bf 1})\mu,
\end{equation}
where $p$ is a Lipschitz continuous function from ${\cal M}_s(\MM)$ to ${\cal P}(\MM)$ which satisfies $p(\mu)=\mu$ on ${\cal P}(\MM)$. For instance, since $\mu\mapsto \mu_+$ is $1$-Lipschitz for $\|.\|_{TV}$, one easily checks that for any $\mu_0\in {\cal P}(\MM)$
$$ p(\mu)=\frac{\mu_+}{\|\mu_+\|_{TV}\vee 1}+(1-\|\mu_+\|_{TV})_{+}\mu_0,$$
has the required properties. Now, by \Cref{lem:pimulip}, it is easily seen that
\begin{equation}\label{eq:linear333}
\mu\mapsto \mu {\cal A}_{p(\mu)}-(\mu {\cal A}_{p(\mu)}{\bf 1})\mu,
\end{equation}
is Lipschitz continuous for $\|.\|_{TV}$ on ${\cal M}_s(\MM)$ so that the Cauchy-Lipschitz theorem ensures existence and uniqueness of solutions on ${\cal M}_s(\MM)$ of \eqref{eq:extendedode}. However, at this stage, we do not know if a solution starting from an initial condition in ${\cal P}(\MM)$ remains in ${\cal P}(\MM)$. 

\noindent To this end, let us consider the ODE
$$
\dot{\varphi}=\varphi_{t} {\cal A}_{p(\mu_t)}, \quad t\ge0,
$$
in $({\cal B}(\MM),\|.\|_\infty)$ (which is a Banach space). The function $\varphi\mapsto \varphi {\cal A}_{p(\mu_t)}$ is Lipschitz with respect to $\|.\|_\infty$, uniformly in $t$ since by \cref{lem:pimulip},
\begin{equation}\label{eq:amutordi}
\sup_{\mu\in{\cal P}(\MM)} \|{\cal A}_\mu\|_\infty=\sup_{\mu\in{\cal P}(\MM)}\sup_{ \|f\|_\infty\le 1} \|{\cal A}_\mu f\|_\infty=\sup_{\mu\in{\cal P}(\MM)} \|{\cal A}_\mu {\bf 1}\|_\infty<+\infty.
\end{equation}
Thus, the equation admits a unique solution $(\varphi_t)_{t\ge0}$ starting from ${\rm Id}$.  For a given $\mu_0\in{\cal P}(\MM)$, set
\begin{equation}\label{eq:rhotratio}
\rho_t:=\frac{\mu_0\varphi_t}{\mu_0\varphi_t{\bf 1}},
\end{equation}
Since $\mu_0\varphi_t{\bf 1}\ge 1$ for every $t\ge0$, $\rho_t$ is well-defined, differentiable on $\ER_+$ and
\begin{equation}\label{eq:rhotlinear}
\dot{\rho}_t=\frac{\mu_0\varphi_t{\cal A}_{p(\mu_t)}}{\mu_0\varphi_t{\bf 1}}-\frac{(\mu_0\varphi_t{\cal A}_{p(\mu_t)}{\bf 1})\mu_0\varphi_t}{(\mu_0\varphi_t{\bf 1})^2}=\rho_t {\cal A}_{p(\mu_t)}-(\rho_t {\cal A}_{p(\mu_t)}{\bf 1})\rho_t.
\end{equation}
The map $\rho\mapsto \rho {\cal A}_{p(\mu_t)}-(\rho {\cal A}_{p(\mu_t)}{\bf 1})\rho$ being uniformly (in $t$) Lipschitz  for the TV-norm, $(\rho_t)_{t\ge0}$ is thus the unique solution (starting from $\mu_0$) of the (linear) ODE $\dot{\rho}_t=\rho_t {\cal A}_{p(\mu_t)}-(\rho_t {\cal A}_{p(\mu_t)}{\bf 1})\rho_t$. But, by construction, the solution $(\mu_t)_{t\ge0}$ of 
the non linear equation \eqref{eq:linear333} is also a solution of the linear ODE $\dot{\rho}_t=\rho_t {\cal A}_{p(\mu_t)}-(\rho_t {\cal A}_{p(\mu_t)}{\bf 1})\rho_t$. Thus, $\rho_t=\mu_t$ and this implies that $(\mu_t)_{t\ge0}$ lives in ${\cal P}(\MM)$ since $(\rho_t)_{t\ge0}$ does (by its very definition \eqref{eq:rhotratio}). As a consequence, $p(\mu_t)=\mu_t$ for all $t\ge0$ and this concludes the proof of this statement.\\

$(ii)$ Since $\mu\mapsto -\mu+\Pi_\mu$ is Lipschitz for the TV-distance (by \cref{lem:pimulip}), uniqueness follows from a classical Gronwall argument. For the existence, we consider 
$\nu_t=\mu_{\tau(t)}$ where $(\mu_t)$ is a solution of \eqref{def:mut45} starting from $\nu_0$ and $\tau$ is a solution of \eqref{eq:tauu}. Before going further, let us check that  \eqref{eq:tauu} has a solution. To this end, we remark that the map 
$\tau\mapsto \mu_{\tau} {\cal A}_{\mu_{\tau}}{\bf 1}$ is Lipschitz since
 $\mu\mapsto {\cal A}_{\mu} {\bf 1}$ is Lipschitz (by \cref{lem:basics}) and $\tau \mapsto \mu_\tau$ also is for the TV-distance. Actually, by \eqref{eq:linearizedmut},
 $$ \forall 0\le \tau_1\le \tau_2, \quad \|\mu_{\tau_2}-\mu_{\tau_1}\|_{TV}\le 2(\sup_{\mu\in{\cal P}(\MM)}\mu {\cal A}_\mu {\bf 1}) (\tau_2-\tau_1)\le C (\tau_2-\tau_1),$$
 by \eqref{eq:amutordi}.  Thus, $\tau\mapsto( \mu_{\tau} {\cal A}_{\mu_{\tau}}{\bf 1})^{-1}$ is also Lipschitz since $\mu{\cal A}_\mu {\bf 1}\ge 1$. As a consequence, \eqref{eq:tauu} has a  unique solution and by construction, 
 $$\forall t\ge0, \quad \dot{\nu}_t=\tau'(t)\dot{\mu}_{\tau(t)}=\frac{\mu_{\tau(t)} {\cal A}_{\mu_{\tau(t)}}}{\mu_{\tau(t)} {\cal A}_{\mu_{\tau(t)}}{\bf 1}}-\mu_{\tau(t)}=\Pi_{\nu_t}-\nu_t.$$
Finally, since ${\cal A}_\mu 1\ge 1$ and $c=\sup_{\mu} \|{\cal A}_\mu {\tcr{\bf 1}}\|_\infty<+\infty$, 
\begin{equation}\label{eq:cundiffeo}
\forall t\ge 0, \quad, c^{-1}\le \tau'(t)\le 1,
\end{equation}
 which in turn implies the ${\cal C}^1$-diffeomorphism property.

\end{proof}
The aim is now to be able to study the asymptotic behavior of $(\mu_t)_{t\ge0}$ defined by \eqref{def:mut45} and to deduce some convergence properties for $(\nu_t)_{t\ge0}$ by time change (given by $\tau$ given by \eqref{eq:tauu}). To this end, we introduce a general continuous function $\alpha=(\alpha_t)_{t\ge0}$  with values in ${\cal P}(\MM)$ and consider the family of operators $(R^{\alpha}_{t,s})_{s\le t}$
\begin{equation}
    R^{\alpha}_{t,s}f:=\frac{\KK^{\alpha}_{t,s}(fg^{\alpha}_{s})}{g^{\alpha}_{t}}
\end{equation}
where $\KK^{\alpha}_{t,s}$ is the unique solution\footnote{By  \eqref{eq:amutordi}, which holds under \ref{condhzero} and \ref{condhun}, existence and uniqueness hold.} at time $t$ starting from $\bar{\varphi}={\rm Id}$ at time $s$ 
 to 
\begin{equation}\label{ode:kts}
\begin{cases}
\dot{\varphi}_v=\varphi_v {\cal A}_{\alpha_v},& v \in[s,t]\\
\varphi_s=\bar{\varphi},
\end{cases}
\end{equation}
and 
$$ g^\alpha_t= \KK^{\alpha}_{t,0} {\bf 1}.$$ 
A crucial point for the sequel is the cocycle property for  $(\KK^{\alpha}_{t,s})_{s<t}$:  $$\forall\; 0\le  u\le s\le t,\quad K_{t,u}=K_{t,s}\circ K_{s,u}.$$
This property is a direct consequence of existence and uniqueness of solutions to the ODE \eqref{ode:kts}. This in turn implies that $(R^{\alpha}_{t,s})_{s<t}$ also has the cocycle property:
\begin{equation}\label{eq:cocycle}
\forall\; 0\le  u\le s\le t,\quad R_{t,u}=R_{t,s}\circ R_{s,u}.
\end{equation}
Also note at this stage that  $(\mu_t)_{t\ge0}$ defined by \eqref{def:mut45} and  $(\KK^{\alpha}_{t,s})_{s\le t}$ are connected by the following relation:
\begin{equation}\label{eq:apreschgttemps}
    \mu_t= \frac{\mu_0 \KK^{\mu_{\bullet}}_t}{\mu_0 \KK^{\mu_{\bullet}}_t {\bf 1}}= R_{t}^{\mu_{\bullet}}\frac{ g^{\mu_{\bullet}}_t}{\mu_0 g^{\mu_{\bullet}}_t},
\end{equation}
where $\KK_t^\alpha=\KK_{t,0}^\alpha$ and $R_t^\alpha=R_{t,0}^\alpha$. In the following, we thus introduce for a given $\alpha:=(\alpha_t)_{t\ge0}$ the map $\mu\mapsto  \Phi^{\alpha}_{t}(\mu)$ defined by
\begin{equation}
    \Phi^{\alpha}_{t}(\mu):=\frac{\mu \KK^{\alpha}_{t}}{\mu \KK^{\alpha}_{t}{\bf 1}}.
\end{equation}
%

%
%

We have the following result:
\begin{prop}\label{prop:condsignalphaI} Assume \ref{condhzero} and \ref{condhun} and that,
\begin{equation}\label{eq:conditioncontractralphaaa}
\|\delta_{x}R^{\alpha}_{t}-\delta_{y}R^{\alpha}_{t}\|_{\TV}\overset{t\to+\infty}{\longrightarrow}0 \quad\textnormal{uniformly in $(x,y,\alpha)\in \MM\times\MM\times {\cal F}(\ER_+,{\cal P}(E))$}.
\end{equation}
Then,
    \begin{itemize}
    \item[(i)] There exists a map $\alpha\mapsto \mu^\star_\alpha$ from ${\cal C}(\ER_+,{\cal P}(\MM))$ to ${\cal P}(\MM)$ such that
        \begin{equation}
        \sup_{\mu,\alpha}\|\mu R^{\alpha}_{t}-\mu^\star_\alpha\|_{\TV}\overset{t\to+\infty}{\longrightarrow}0.
    \end{equation}
  \item[(ii)]  The convergence result also holds for $\Phi^{\alpha}_{t}(\mu)$:
    \begin{equation}
        \sup_{\mu,\alpha}\|\Phi^{\alpha}_{t}(\mu)-\mu^\star_\alpha\|_{\TV}\overset{t\to+\infty}{\longrightarrow}0.
    \end{equation}
    \end{itemize}
\end{prop}
\begin{Rq} In this result, the limit does not depend on the starting distribution $\mu$. Unfortunately, when we apply to $\alpha=(\mu_{t})_{t\ge0}$, \emph{i.e.} on the distribution of the solution itself, the limit will \emph{a priori} depend on $\mu_0$.
\end{Rq}

\begin{proof}
  \noindent  $(i)$ Let $f:\MM\rightarrow\ER$ denote a  bounded function with $\|f\|_\infty\le 1$. For all $\mu,\nu\in\mathcal{P}(\mathcal{E})$, we have
    \begin{align}
        |\mu R^{\alpha}_{t}f-\nu R^{\alpha}_{t}f|&=\bigg|\int\bigg(R^{\alpha}_{t}f(x)-R^{\alpha}_{t}f(y)\bigg)\mu(dx)\nu(dy)\bigg|\\
        &\leq\sup_{x,y}|R^{\alpha}_{t}f(x)-R^{\alpha}_{t}f(y)|.\nonumber
    \end{align}
    As $\|\delta_{x}R^{\alpha}_{t}-\delta_{y}R^{\alpha}_{t}\|_{\TV}\overset{t\to+\infty}{\longrightarrow}0$ uniformly in $(x,y,\alpha)$, there exists $(\Delta_{t})_{t\ge0}$ independent of $x,y,\alpha,$ and $f$ (with $\|f\|_\infty\le 1$) such that 
    \begin{equation}\label{eq:contractralpha}
        |\mu R^{\alpha}_{t}f-\nu R^{\alpha}_{t}f|\leq\Delta_{t}\overset{t\to+\infty}{\longrightarrow}0.
    \end{equation}
  Using the cocycle property \eqref{eq:cocycle}, we deduce that
    \begin{equation}\label{eq:cauchybound}
        \|\mu R^{\alpha}_{t+s}-\mu R^{\alpha}_{t}\|_{\TV}=\|(\mu R^{\alpha}_{t+s,t})R^{\alpha}_{t}-\mu R^{\alpha}_{t}\|_{\TV}\leq\Delta_{t}.
    \end{equation}
    and hence,  $(\mu R^{\alpha}_{t})_{t\geq0}$ is a Cauchy sequence for the $\|.\|_{\TV}$-norm. For every $(\alpha,\mu)$, this thus converges to a probability $\mu^\star_{\alpha}$ which does not depend on $\mu$ thanks to \eqref{eq:contractralpha}. The bound in \eqref{eq:cauchybound} being uniform in $(\alpha,\mu)$, the convergence to this $\mu^\star_{\alpha}$  is uniform, that is:
\begin{equation}\label{eq:unifconvRalpha}
    \sup_{\mu,\alpha}\|\mu R^\alpha_t-\mu^\star_\alpha\|_{\TV}\le \Delta_t\xrightarrow{t\rightarrow+\infty}0.
    \end{equation}
    
  \noindent $(ii)$  For a probability distribution $\nu$ and a measurable function $f:\MM\rightarrow\ER$, we have
  $$
  (\Phi^{\alpha}_{t}(\mu)-\nu)f=\frac{1}{\mu g^{\alpha}_{t}}\left[\mu (g^{\alpha}_{t}\left(R^{\alpha}_{t}f-\nu (f))\right)\right]=\mathbb{E}[R^{\alpha}_{t}f(Z)-\nu (f)]
  $$
where $\PE_Z=\frac{\mu(.\times g_t^\alpha)}{ \mu g_t^\alpha}$. Thus, using \eqref{eq:unifconvRalpha}, we deduce that
$$
  \|\Phi^{\alpha}_{t}(\mu)-\mu^\star_\alpha\|_{\TV}\le \sup_{\nu\in{\cal P}(\MM)}\|\nu R^\alpha_t-\mu^\star_\alpha\|_{\TV}\le \Delta_t\xrightarrow{t\rightarrow+\infty}0.
  $$
\end{proof}
Now, we  want to exhibit some sufficient conditions which ensure \eqref{eq:conditioncontractralphaaa}.
\begin{prop}\label{prop:condsignalpha}
 Assume \ref{condhzero} and \ref{condhun} and that there exist $\varepsilon>0$, $c>0$ and a probability distribution $\Psi$ such that for all $n$ and $\alpha$, $\KK^{\alpha}_{n+1,n}\geq\varepsilon\Psi$, $\|\KK^{\alpha}_{n+1,n}\|_{\infty}\leq c$ and 
    \begin{equation}\label{eq:condpsignaplha}
        \sum_{n}\frac{\Psi(g^{\alpha}_{n})}{\|g^{\alpha}_{n}\|_{\infty}}=+\infty.
    \end{equation}

    Then, Condition \eqref{eq:conditioncontractralphaaa} holds true.
\end{prop}

\begin{proof}
\textit{Step 1}: We show that under the assumptions, 
$$R^{\alpha}_{n+1,n}\geq\varepsilon_{n}\Psi^{\alpha}_{n},$$
where $(\varepsilon_n)_{n\ge0}$ is a sequence of positive numbers with $\sum \varepsilon_n=+\infty$ and $\Psi^{\alpha}_{n}\in{\cal P}(\MM)$ (depending on $n$ and $\alpha$): for every non-negative function $f$, we have
    \begin{equation*}
        R^{\alpha}_{n+1,n}f=\frac{\KK^{\alpha}_{n+1,n}(fg^{\alpha}_{n})}{g^{\alpha}_{n+1}}.
    \end{equation*}
    Since $g^{\alpha}_{n+1}=\KK^{\alpha}_{n+1,n}g^{\alpha}_{n}$, it follows that
    \begin{equation*}
        R^{\alpha}_{n+1,n}f\geq\frac{\varepsilon\Psi(fg^{\alpha}_{n})}{c\|g^{\alpha}_{n}\|_{\infty}}=\varepsilon_n\Psi_n^\alpha,
    \end{equation*}
with
    \begin{equation*}
        \varepsilon_{n}=\frac{\varepsilon\Psi(g^{\alpha}_{n})}{c\|g^{\alpha}_{n}\|_{\infty}}\quad\textnormal{and}\quad\Psi^{\alpha}_{n}(f)=\frac{\Psi(fg^{\alpha}_{n})}{\Psi(g^{\alpha}_{n})}.
    \end{equation*}
Assumption \eqref{eq:condpsignaplha} finally ensures that  $\sum \varepsilon_n=+\infty$.\\

\noindent \textit{Step 2}: We prove the statement. By Step 1,
\begin{equation*}
        \delta_{x}R^{\alpha}_{n+1,n}=\varepsilon_{n}\Psi^{\alpha}_{n}+(1-\varepsilon_{n})Q^{\alpha}_{n+1}(x,\cdot),
    \end{equation*}
    where $Q^{\alpha}_{n+1}$ is a Markov kernel. We deduce that
    \begin{equation*}
        \mu R^{\alpha}_{n+1,n}-\nu R^{\alpha}_{n+1,n}=(1-\varepsilon_{n})(\mu Q^{\alpha}_{n+1}-\nu Q^{\alpha}_{n+1})
    \end{equation*}
    Since for a given function $f$, $\|Q^{\alpha}_{n+1} f\|_\infty\le \|f\|_\infty$, we get
    \begin{equation*}
        \|\mu R^{\alpha}_{n+1,n}-\nu R^{\alpha}_{n+1,n}\|_{\TV}= (1-\varepsilon_{n})\|\mu Q^{\alpha}_{n+1}-\nu Q^{\alpha}_{n+1}\|_{\TV}\leq(1-\varepsilon_{n})\|\mu-\nu\|_{\TV}.
    \end{equation*}
    Thus, for all $k\in\{0,\ldots,n\}$, we have
    \begin{equation*}
        \|(\mu R^{\alpha}_{n+1,k+1})R^{\alpha}_{k+1,k}-(\nu R^{\alpha}_{n+1,k+1})R^{\alpha}_{k+1,k}\|_{\TV}\leq(1-\varepsilon_{k})\|\mu R^{\alpha}_{n+1,k+1}-\nu R^{\alpha}_{n+1,k+1}\|_{\TV},
    \end{equation*}
    and an induction leads to
    \begin{equation*}
        \|\mu R^{\alpha}_{n}-\nu R^{\alpha}_{n}\|_{\TV}\leq\prod_{k}(1-\varepsilon_{k})\|\mu-\nu\|_{\TV}\le 2\prod_{k}(1-\varepsilon_{k}) \overset{n\to+\infty}{\longrightarrow}0,
    \end{equation*}
    since $\sum\varepsilon_{n}=+\infty$. Furthermore, the convergence is uniform in $(\nu,\mu,\alpha)$. This uniform convergence extends to $t\in\ER_+$ by noting that
    \begin{equation*}
       \|\mu R^{\alpha}_{t}-\nu R^{\alpha}_{t}\|_{\TV}=\|\mu R^{\alpha}_{t,\lfloor t\rfloor })R^{\alpha}_{\lfloor t\rfloor}-(\nu R^{\alpha}_{t,\lfloor t\rfloor})R^{\alpha}_{\lfloor t\rfloor}\|_{\TV}.
    \end{equation*}
\end{proof}
We are now ready to state the main result of this section:
\begin{prop}\label{prop:convuniformemustar}
    Assume \ref{condhzero}, \ref{condhun} and \ref{condhquatre}. Then, 
    \begin{itemize}
    \item[(i)]
     \begin{equation}
        \sup_{\mu,\alpha}\|\Phi^{\alpha}_{t}(\mu)-\mu^\star_\alpha\|_{\TV}\overset{t\to+\infty}{\longrightarrow}0.
    \end{equation}
    \item[(ii)] For any initial condition $\nu_0\in {\cal P}(\MM)$, there exists a probability ${\cal I}_{\nu_0}$ such that the unique solution  to $\dot{\nu}=-\nu+\Pi_\nu$ converges to ${\cal I}_{\nu_0}$ as $t\rightarrow+\infty$ for the $TV$-distance. Furthermore, the convergence is uniform in the initial condition:
    $$\sup_{\nu_0\in{\cal P}(\MM)}\|\nu_t^{\nu_0}-{\cal I}_{\nu_0}\|_{TV}\xrightarrow{t\rightarrow+\infty}0.$$
    \end{itemize}
    \end{prop}
    \begin{Rq} Following carefully the proof, we in fact obtained the existence of some positive $C$ and $\beta$ such that
     \begin{equation}
        \sup_{\mu,\alpha}\|\Phi^{\alpha}_{t}(\mu)-\mu^\star_\alpha\|_{\TV}\le C e^{-\beta t}.
    \end{equation}
    \end{Rq}
 \begin{proof}
\emph{(i)}  By \Cref{prop:condsignalphaI}\emph{(ii)} and \Cref{prop:condsignalpha}, it is enough to check that \eqref{eq:condpsignaplha} holds true.  We begin by giving a series expansion of the solution of 
\begin{equation}
    \dot{\varphi}_{t}=\varphi_{t}{\cal A}_{\alpha_{t}}.
\end{equation}
Let $s_0\in\ER_+$:
\begin{align*}
    \varphi_{s_{0}}&=\varphi_{0}+\int_{0}^{s_{0}}\varphi_{s_{1}}{\cal A}_{\alpha_{s_{1}}}ds_{1}=\varphi_{0}+\int_{0}^{s_{0}}\bigg(\varphi_{0}+\int_{0}^{s_{1}}\varphi_{s_{2}}{\cal A}_{\alpha_{s_{2}}}ds_{2}\bigg){\cal A}_{\alpha_{s_{1}}}ds_{1}\\
    &=\varphi_{0}+\int_{0}^{s_{0}}\varphi_{0}{\cal A}_{\alpha_{s_{1}}}ds_{1}+\int_{0}^{s_{0}}\int_{0}^{s_{1}}\varphi_{s_{2}}{\cal A}_{\alpha_{s_{2}}}{\cal A}_{\alpha_{s_{1}}}ds_{2}ds_{1}.
\end{align*}
An induction then leads to:  
\begin{align*}
\forall n\geq 1,\quad   \varphi_{s_{0}}=\varphi_{0}&+\sum_{k=1}^{n}\int_{0}^{s_{0}}\cdots\int_{0}^{s_{k-1}}\varphi_{0}{\cal A}_{\alpha_{s_{k}}}\cdots {\cal A}_{\alpha_{s_{1}}}ds_{k}\cdots ds_{1}+r_n(s_0),
   \end{align*}
   where 
$$   r_n(s_0)=\int_{0}^{s_0}\cdots\int_{0}^{s_{n}}\varphi_{s_{n+1}}{\cal A}_{\alpha_{s_{n+1}}}\cdots {\cal A}_{\alpha_{s_{1}}}ds_{n+1}\cdots ds_{1}.$$
Under the assumptions,  $\sup_{\mu\in{\cal P}(\MM)}\|{\cal A}_\mu\|_\infty<+\infty$. This implies that
$$ \|r_n(s_0)\|_\infty\le \sup_{s\in[0, s_0]}\|\varphi_s\|_\infty \frac{(cs_0)^{n+1}}{(n+1)!}\le \|\varphi_0\|_\infty e^{c s_0}\frac{(cs_0)^{n+1}}{(n+1)!}, $$
where the second inequality follows from Gronwall's lemma. Thus, for any $s_0$, $r_n(s_0)\xrightarrow{n\rightarrow+\infty}0$ and hence, 
\begin{align}\label{eq:semiexplicit}
   K_{s_0,0}^{\alpha}={\rm Id}+\sum_{k\ge1}\int_{0}^{s_{0}}\cdots\int_{0}^{s_{k-1}}{\cal A}_{\alpha_{s_{k}}}\cdots {\cal A}_{\alpha_{s_{1}}}ds_{k}\cdots ds_{1}.
   \end{align}
Note that when $\mu\mapsto {\cal A}_{\mu}$  is constant, one retrieves the classical exponential solution of the system.\\

\noindent We now use the very definition of ${\cal A}_\mu$ to get for any $k\ge1$:
$${\cal A}_{\alpha_{s_{k}}}\cdots {\cal A}_{\alpha_{s_{1}}}=\sum_{(n_1,\ldots,n_{k})\in \mathbb{N}^{k+1}} K_{\alpha_{s_{k},\partial}}^{n_{k}}\cdots K_{\alpha_{s_{1},\partial}}^{n_1}.$$
Thus, using \ref{condhquatre} in the second line, we get
\begin{align*}
\Psi(g_t^\alpha)&=\Psi(K_{t,0}^\alpha{\bf 1})=1+\sum_{k\ge1}\sum_{(n_1,\ldots,n_{k})\in \mathbb{N}^{k+1}}\int_{0}^{s_{0}}\cdots\int_{0}^{s_{k-1}}\Psi(K_{\alpha_{s_{k}},\partial}^{n_{k}}\cdots K_{\alpha_{s_{1}},\partial}^{n_1}{\bf 1})ds_{s_k}\cdots ds_{1}\\
&\ge c\left[1+\sum_{k\ge1}\sum_{(n_1,\ldots,n_{k})\in \mathbb{N}^{k+1}}\int_{0}^{s_{0}}\cdots\int_{0}^{s_{k-1}}K_{\alpha_{s_{k}},\partial}^{n_{k}}\cdots K_{\alpha_{s_{1}},\partial}^{n_1}{\bf 1}ds_{k}\cdots ds_{1}\right]\\
&\ge c g_t^{\alpha}.
\end{align*}
In particular,
$$\inf_{n,\alpha}\frac{\Psi(g_n^\alpha{\bf 1})}{\|g_n^{\alpha}\|_\infty}\ge c>0,$$
which (trivially) implies \eqref{eq:condpsignaplha}.\\

\noindent \emph{(ii)} To prove this statement, let us recall that by \Cref{lem:changevariable} (and the related notations), $(\mu_t)$ defined by $\mu_t=\nu_{\tau^{-1}(t)}^{\nu_0}$ satisfies \eqref{def:mut45} (where $\tau^{-1}$ denotes the inverse of the bijective function $\tau$). By  \eqref{eq:apreschgttemps}, one deduces that 
$$\mu_t=\Phi_t^{\mu_{\bullet}}(\nu_0).$$
We thus apply  \emph{(i)} with $\alpha=\mu_{\bullet}$. Note that by the uniqueness property given in \cref{lem:changevariable}\emph{(i)}, $(\mu_t)_{t\ge0}$  is uniquely determined by its initial condition $\nu_0$. We can thus denote the probability $\mu_\alpha^\star$ of \emph{(i)} by ${\cal I}_{\nu_0}$ and deduce that
$$\sup_{\nu_0\in{\cal P}(\MM)}\|\mu_t^{\nu_0}-{\cal I}_{\nu_0}\|_{TV}\xrightarrow{t\rightarrow+\infty}0.$$
Finally, since the time change $\tau$ satisfies \eqref{eq:cundiffeo} where $c$ is independent of $\nu_0$, the above uniform convergence also applies to $(\nu_t^{\nu_0})_{t\ge0}$.
 \end{proof}

\section{Proof of Theorem \ref{theo:discret}}\label{subsecseptun}
\emph{(i)} We first check that \ref{condhzero}, \ref{condhun} and \ref{condhdeux} imply \ref{condpzero}, \ref{condpun} and \ref{condpdeux} when 
$$\kmu= K_{\mu,\partial}+\delta_\mu(.)\mu.$$
The uniqueness of the invariant distribution $\Pi_\mu$ is a direct consequence of \ref{condhtrois}. The fact that $\mu\mapsto \Pi_\mu$ is Lipschitz is established in \Cref{lem:pimulip}. Thus, \ref{condpzero} is true. Assumption \ref{condpun} is a direct consequence of \ref{condhtrois}. Finally, by \ref{condhdeux}, $\mu\mapsto K_{\mu,\partial}$ and $\mu\mapsto \delta_{\mu}$ are clearly Lipschitz for the TV-norm so that $\mu\mapsto \kmu$ is finally Lipschitz continuous. Finally, by \eqref{eq:munrecursive}, $\|\mu_{n+1}-\mu_n\|_{TV}\le 2\gamma_{n+1}$, so that $(X_n,\mu_n)$ satisfies the assumptions of  \cref{prop:asymptoticpseudotrajectory}. 

Then, the property $(i)$ is a classical consequence of the conclusion of  this proposition (\cref{prop:asymptoticpseudotrajectory}), of the tightness assumption on $(\mu_n)$,  and of the theory of asymptotic pseudo-trajectories. For the sake of completeness, we recall insights and references. First, using again that  $\|\mu_{n+1}-\mu_n\|_{TV}\le 2\gamma_{n+1} $, one can check that $(\tilde{\mu}_t)_{t\ge0}$ is Lipschitz for the $TV$-distance (which metrizes the weak topology). Then, owing to the tightness of $(\mu_n)_{n\ge1}$ and to the ``shifted'' construction of the sequence $(\tilde{\mu}^{(n)})_{n\ge0}$, one deduces the tightness of $(\tilde{\mu}^{(n)})_{n\ge0}$ on ${\cal C}(\ER_+,{\cal P}(\MM))$ from an Ascoli argument (see \cite[Theorem 3.2]{B99} or \cite[Proof of Proposition 3.5]{panloup_reygner} for similar arguments).\\

Second, to show that every weak limit is a solution to $\dot{\nu}=F(\nu)$, we first deduce from the construction that 
$$ \forall t\ge0, \quad \tilde{\mu}^{(n)}_t= \tilde{\mu}^{(n)}_0+\int_0^t F(\tilde{\mu}^{(n)}_{N(n,s)}) ds+\int_0^t \varepsilon_{N(n,s)+1} ds.$$
Noting that for a given bounded function $f$, 
$$\int_0^t \varepsilon_{N(n,s)+1}(f) ds=|\sum_{k=n+1}^{N(n,t)+1} \gamma_k\left(f(X_k)-\Pi_{\mu_k}(f)\right)+O(\gamma_{n}^{-1}),$$
we deduce from \Cref{prop:asymptoticpseudotrajectory} that for any $T>0$,
$$\lim_{n\rightarrow+\infty} \sup_{t\le T} \left|\int_0^t \varepsilon_{N(n,s)+1}(f) ds\right|=0\quad a.s.$$
Combined with the fact that $F$ is Lipschitz continuous (by \cref{lem:pimulip}), this leads to the following property: for all bounded measurable $f:\MM\rightarrow\ER$, for all $T>0$,
$$\lim_{n\rightarrow+\infty} \sup_{t\le T} \left|\tilde{\mu}^{(n)}_t(f)-\nu_t^{\mu_n})\right|=0\quad a.s.,$$
where $\nu_t^{\mu}$ denotes the unique solution of \eqref{eq:nutdyn} starting from $\mu$ (existence and uniqueness are given by \cref{lem:changevariable}\emph{(ii)}). This property then implies the result (see \cite[Theorem 3.5]{BLR02} for a very similar result).\\

\emph{(ii)} From the first statement, we know that $a.s.$, any weak limit of $(\mutilde^{(n)})_n$ is a solution to $\dot{\nu}=-\nu+\Pi_\nu$. Let $(\mutilde^{(n_k)})_{k\ge1}$ denote such a convergent subsequence of $(\mutilde^{(n)})_{n\ge1}$ and let $\nu^{\infty}$ denote its limit. The process $(\nu^{\infty}_t)_{t\ge0}$ is thus a solution to $\dot{\nu}=-\nu+\Pi_\nu$ starting from $\nu^{\infty}_0$. By \Cref{prop:convuniformemustar}, there exists ${\cal I}_{\nu^{\infty}_0}$ 
such that 
$$\|\nu^{\infty}_t-{\cal I}_{\nu^{\infty}_0}\|_{TV} \xrightarrow{t\rightarrow+\infty}0.$$
Now, for some given $n\ge1$ and $\tau\ge1$, recall that $N(n,\tau):=\max\{k\ge 1, \sum_{\ell=k+1}^n\gamma_{\ell}\le \tau\}$. By construction, there exists a sequence $(\tau_n)_{n\ge1}$ such that
 $$  \tau_n\le \tau\le \tau_n+\gamma_{n+1}\quad \textnormal{and}\quad \mu_n=\mutilde^{(N(n,\tau))}_{\tau_n}.$$
 At the price of extracting again a subsequence, $(\mutilde^{(N(n_k,\tau))})_{k\ge1}$ also converges to a solution to $\dot{\nu}=-\nu+\Pi_\nu$. Denote it by $(\nu^{\infty,\tau}_t)_{t\ge0}$. Since $(\tau_n)_{n\ge1}$ converges to $\tau$ and since the limiting function is continuous, we have:
$$ \nu^{\infty,\tau}_{\tau}=\lim_{n\rightarrow+\infty}\mu_n=\nu_0.$$
By uniqueness of solution, this implies that for $t\ge \tau$,
$$\nu^{\infty,\tau}_{t}=\nu^{\infty}_{t-\tau}.$$
As a consequence, $(\nu^{\infty,\tau}_{t})_{t\ge0}$ has the same limit as $t\rightarrow+\infty$, that is ${\cal I}(\nu_0)$. But since the convergence to ${\cal I}(\nu_0)$ is uniform in the initial condition, for any $\varepsilon$, we can fix $\tau_\varepsilon>0$ such that 
$$\|\nu^{\infty,\tau_{\varepsilon}}_{t}-{\cal I}(\nu_0)\|_{TV}\le \varepsilon\quad \forall t\ge \tau_\varepsilon.$$
Thus, for any $\varepsilon>0$, 
$$\|\nu_0-{\cal I}(\nu_0)\|_{TV}=\|\nu^{\infty,\tau_{\varepsilon}}_{\tau_\varepsilon}-{\cal I}(\nu_0)\|_{TV}\le \varepsilon.$$
Finally, $\nu_0={\cal I}(\nu_0)$ which implies that the process is stationary. As a consequence, $\dot\nu_0=0$, which in turn implies that $\nu_0=\Pi_{\nu_0}$. Thus, $\nu_0$ is a fixed point for the map $\mu\mapsto \Pi_\mu$, \emph{i.e.} a QSD by \cref{lem:characQSD}.


\section{Tightness} \label{sec:proof3}
We proce \Cref{prop:tightnessbis} which provides sufficient conditions for $(\mu_n)$ to be $a.s.$ tight.

\begin{proof}[Proof of  \Cref{prop:tightnessbis}] The proof is divided into two steps.\\

\textit{Step 1.} We prove that $\exists\,\varepsilon_0>0$ such that $\forall\,\varepsilon\in(0,\varepsilon_0]$, $\sup_n \{\ES[\mu_n(V_\varepsilon)]+\ES[V_\varepsilon(X_n)]\}<+\infty.$\\

\noindent By the recursive formula \eqref{eq:munrecursive} and the definition of the dynamics of $(X_n)$, we have:
\begin{equation}\label{eq:coupleunvnstoch}\begin{cases}
\mu_{n+1}(V_\varepsilon)=\mu_n(V_\varepsilon)(1-\gamma_{n+1})+\gamma_{n+1} V_\varepsilon(X_{n+1})\\
\ES[V_\varepsilon(X_{n+1})|{\cal F}_n]=K_{\mu_n,\partial} V_\varepsilon(X_n)+\delta(X_n) \mu_n(V_\varepsilon).
\end{cases}
\end{equation}
Thus, setting   $u_n=\ES[\mu_n(V_\varepsilon)]$ and $v_n=\ES[V_\varepsilon(X_n)]$ (depending on $\varepsilon$ even if the notation does not), we deduce from \ref{condhvdeux} that:
\begin{equation}\label{eq:coupleunvn}
\begin{cases}
u_{n+1}\le u_n(1-\gamma_{n+1})+\gamma_{n+1} v_{n+1}\\
v_{n+1}\le \varepsilon v_n+\beta_\varepsilon+\|\delta\|_{\infty} u_n.
\end{cases}
\end{equation}
Setting $L_n=u_n+\lambda \gamma_n v_n$ where $\lambda>0$, we obtain:
\begin{align*}
L_{n+1}\le u_n\left(1 -(1-(1+\lambda)\|\delta\|_\infty)\gamma_{n+1}\right)+ \lambda \gamma_{n+1} v_n \varepsilon \left(1+\lambda^{-1}\right)+ \beta_{\varepsilon}(1+\lambda)\gamma_{n+1}.
\end{align*}
{Since $\|\delta\|_\infty<1$},  one can fix $\lambda$ sufficiently small such that
$$\rho_\lambda=1-(1+\lambda)\|\delta\|_\infty>0.$$
Then, if $\varepsilon_0=\frac{1}{2}(1+\lambda^{-1})^{-1}$, we get for all $\varepsilon\in(0,\varepsilon_0]$ (using that $(\gamma_n)$ is non-increasing):
\begin{align*}
L_{n+1}&\le u_n\left(1 -\rho_\lambda \gamma_{n+1}\right)+ \frac{1}{2} \lambda \gamma_{n} v_n + C\gamma_{n+1}\\
&\le L_n\left(1 -\rho_\lambda \gamma_{n+1}\right)+C\gamma_{n+1}\quad \forall n\ge n_0,
\end{align*}
where $C$ is a finite constant and,
$$ n_0:=\inf\{n\in\mathbb{N}, 1 -\rho_\lambda \gamma_{n+1}\ge\frac{1}{2}\}.$$
The integer $n_0$ is certainly finite since $(\gamma_n)$ tends to $0$ as $n\rightarrow+\infty$. Then, an induction leads to 
\begin{align}\label{eq:iterexpo}
L_n\le L_{n_0}\prod_{k=n_0+1}^n (1-\rho_\lambda \gamma_k)+ C\sum_{k=n_0+1}^n \gamma_k  \prod_{\ell=k+1}^n (1-\rho_\lambda \gamma_\ell)\nonumber\\
\le L_{n_0}+C \sum_{k\ge n_0+1} \gamma_k \exp(-\rho_\lambda \sum_{\ell=k+1}^n \gamma_k).
\end{align}
In the second line, we used that $\rho_\lambda \gamma_n\le 1/2$ for $n\ge n_0+1$ and the inequality $\log(1+x)\le x$ for $x>-1$. By an integral-series comparison argument, one checks that (with $t_n=\sum_{k=1}^n\gamma_k$):
\begin{equation}\label{eq:argumentexpoff}
\sum_{k\ge n_0+1} \gamma_k \exp\left(-\rho_\lambda \sum_{\ell=k+1}^n \gamma_k\right)\le e^{-\rho_\lambda t_n} \int_0^{t_n} e^{\rho_\lambda s} ds\le \frac{1}{\rho_\lambda}.
\end{equation}
Thus, checking by induction that $L_{n_0}$ is finite, it follows that
\begin{equation}
\sup_{n\ge 1} L_n<+\infty.
\end{equation}
Thus, $\sup_n u_n<+\infty$ and plugging this control into \eqref{eq:coupleunvn}, we deduce that 
$$ v_{n+1}\le \varepsilon v_n+ C$$
and $\sup_n v_n<+\infty$.\\

\noindent \textit{Step 2.} We prove the result. \\

\noindent By Jensen inequality and \ref{condhvdeux}, we have: for all $\varepsilon>0$,
$$ \forall \mu\in{\cal P}(\MM), \quad K_{\mu,\partial}\sqrt{V_\varepsilon}\le \bar{\varepsilon} \sqrt{V_\varepsilon}+\bar{\beta}_\varepsilon\quad\textnormal{with $\bar{\varepsilon}=\sqrt{\varepsilon}$ and
$\bar{\beta}_\varepsilon=\sqrt{{\beta}_\varepsilon}$.}  $$
Now, by \eqref{eq:coupleunvnstoch} applied with ${\cal V}=\sqrt{V_\varepsilon}$,
\begin{equation}\label{eq:techninc}
\mu_{n+1}({\cal V})\le \mu_n({\cal V})(1-\gamma_{n+1}(1-\|\delta\|_\infty))+\bar{\varepsilon} \gamma_{n+1} {\cal V}(X_n)+\gamma_{n+1} \bar{\beta}_\varepsilon+\gamma_{n+1}\Delta M_{n+1},
\end{equation}
with 
$$
\Delta M_{n+1}={\cal V}(X_{n+1})-\ES[{\cal V}(X_{n+1})|{\cal F}_n]. 
$$
For the sake of readability, we first provide a simple proof case when the assumption holds with $\varepsilon=0$ ($i.e.$ when there is an inf-compact function $V_0$ such that $\sup_{\mu,x} K_{\mu,\partial} V_0(x)<+\infty$).

\noindent \textit{Case $\varepsilon=0$.} Set $\rho=1-\|\delta\|_\infty$ and  $n_0=\inf\{n, \rho\gamma_n\le \frac{1}{2}\}$. An induction leads to 
$$\mu_n({\cal V})\le \mu_n({\cal V})\prod_{k=n_0+1}^{n}(1-\rho \gamma_k)+C \sum_{k=n_0+1}^n \gamma_k \Theta_{n,k}+\sum_{k=n_0+1}^n \gamma_k \Delta M_{k}\Theta_{n,k},$$
with $\rho=1-\|\delta\|_\infty$ and,
$$\Theta_{n,k}=\prod_{\ell=k+1}^n (1-\rho\gamma_k).$$
As in the first part of the proof, we have
$$\forall\; k\in\llbracket n_0, n\rrbracket,\quad \Theta_{n,k}\le e^{-\rho (t_n-t_k)},$$
so that 
$$ \sup_{n\ge n_0+1} \sum_{k=n_0+1}^n \gamma_k \Theta_{n,k}<+\infty.$$
Finally, to prove that 
$$\sup_{n\ge n_0+1} \sum_{k=n_0+1}^n \gamma_k \Delta M_{k}\Theta_{n,k} <+\infty \quad a.s.,$$
we apply \Cref{lem:kron}  with $b_n=e^{\rho t_n}$ and $u_n=\gamma_n \Delta M_n$ and thus, check that the martingale $(N_n)$ defined by
$$ N_n= \sum_{k=1}^n \gamma_k \Delta M_{k}, \quad n\ge 1,$$
is convergent.  But, 
$$\langle N\rangle_n\le \sum_{k=1}^n\gamma_k^2 V_0(X_k),$$
so that by Step $1$, 
$$\ES[\langle N\rangle_\infty]\le C\sum_{k\ge1}\gamma_k^2<+\infty.$$
This concludes the proof.

\noindent \textit{Case $\varepsilon>0$.} We now consider the real assumption \ref{condhvdeux}. We go back to \eqref{eq:techninc} and choose to iterate the control of  ${\cal V}(X_n)$ induced by \eqref{eq:coupleunvnstoch}:
\begin{align*}
\bar{\varepsilon} {\cal V}(X_n)&\le \bar{\varepsilon}^2 {\cal V}(X_{n-1})+ \bar{\varepsilon} \Delta M_n+ \bar{\varepsilon}\bar{\beta}_\varepsilon+\bar{\varepsilon} \mu_{n-1}({\cal V})\\
&\le \bar{\varepsilon}^n {\cal V}(X_0)+ \sum_{k=1}^{n-1}  \bar{\varepsilon}^{n-k}  \left(\Delta M_{k+1}+ \bar{\beta}_\varepsilon+ \mu_{k}({\cal V})\right).
\end{align*}
Hence, plugging into  \eqref{eq:techninc} and setting 
$$\nu_n({\cal V})=\sum_{k=1}^{n-1} \bar{\varepsilon}^{n-1-k} \mu_{k}({\cal V}),$$
we get:
$$
\begin{cases}
\mu_{n+1}({\cal V})\le \mu_n({\cal V})(1-\gamma_{n+1}(1-\|\delta\|_\infty))+ \gamma_{n+1} \left[\bar{\varepsilon}\nu_n({\cal V})+C_\varepsilon+\sum_{k=1}^{n}  \bar{\varepsilon}^{n-k}  \Delta M_{k+1}\right]\\
\nu_{n+1}({\cal V})\le \bar{\varepsilon} \nu_n({\cal V})+\mu_n({\cal V}),
\end{cases}
$$
with $C_\varepsilon={\cal V}(X_0)+(1-\bar{\varepsilon})^{-1}\bar{\beta}_\varepsilon$. Setting $\tilde{L}_n=\mu_{n}({\cal V})+\lambda \gamma_n \nu_{n}({\cal V})$ (where $\lambda$ will be calibrated further), one checks that 
\begin{align*}
 \tilde{L}_{n+1}&\le (1-\gamma_{n+1}(1-\|\delta\|_\infty-\lambda))\mu_n({\cal V})+\lambda \gamma_{n+1}\left(\frac{\bar{\varepsilon}}{\lambda}+\bar{\varepsilon}\right)\nu_n({\cal V})\\
 &+\gamma_{n+1}
\left( C_\varepsilon+ \sum_{k=1}^{n}  \bar{\varepsilon}^{n-k}  \Delta M_{k+1}\right).
\end{align*}
Fix $\lambda$ small enough in such a  way that $\tilde{\rho}_\lambda=1-\|\delta\|_\infty-\lambda>0$. Then choosing $\bar{\varepsilon}$ such that $\frac{\bar{\varepsilon}}{\lambda}+\bar{\varepsilon}<1$, $i.e.$ such that ($\bar{\varepsilon}=\sqrt{\varepsilon}$),
\begin{equation}\label{secondcondeps}
\varepsilon \le \tilde{\varepsilon}_0=(1+\lambda^{-1})^{-2},
\end{equation}
leads to the existence of an integer $\tilde{n}_0$ such that for all $n\ge\tilde{n}_0$,
 $$ \tilde{L}_{n+1}\le(1-\tilde{\rho}_\lambda\gamma_{n+1})\tilde{L}_n+\gamma_{n+1}
\left( C_\varepsilon+ \sum_{k=1}^{n}  \bar{\varepsilon}^{n-k}  \Delta M_{k+1}\right).$$
At the price of replacing $\tilde{n}_0$ by a greater integer, we assume that $\tilde{\rho}_\lambda \gamma_n\le 1/2$ for $n\ge \tilde{n}_0$. Then, a similar induction as in \eqref{eq:iterexpo} leads to:
 \begin{align*}
\tilde{L}_n\le \tilde{L}_{\tilde{n}_0}\exp(-\rho_\lambda (t_n-t_{\tilde{n}_0}))
+ \sum_{k= \tilde{n}_0+1}^n \gamma_k
\left( C_\varepsilon+ \sum_{\ell=1}^{k}  \bar{\varepsilon}^{k-\ell}  \Delta M_{\ell+1}\right) \exp(-\rho_\lambda(t_n-t_{k})),
\end{align*}
where $t_n=\sum_{k=1}^n\gamma_k$. On the one hand, a similar argument as the one of \eqref{eq:argumentexpoff} leads to 
$$\sup_{n\ge\tilde{n}_0+1}\sum_{k= \tilde{n}_0+1}^n \gamma_k C_\varepsilon \exp(-\rho_\lambda(t_n-t_{k}))<+\infty\quad a.s.$$
One the other, we want to show that,
\begin{equation}
\sup_{n\ge\tilde{n}_0+1}e^{-\rho_\lambda t_n}\sum_{k= \tilde{n}_0+1}^n \gamma_ke^{\rho_\lambda t_k} \sum_{\ell=1}^{k} \bar{\varepsilon}^{k-\ell}\Delta M_{\ell+1} <+\infty\quad a.s.
\end{equation}
By \Cref{lem:kron} applied with $b_n= e^{\rho t_n}$ and $u_n= \gamma_n \sum_{\ell=1}^{n} \bar{\varepsilon}^{n-\ell} \Delta M_{\ell+1}$,  it is enough to show that, 
\begin{equation*}
\sum \left(\gamma_k \sum_{\ell=1}^{k} \bar{\varepsilon}^{k-\ell} \Delta M_{\ell+1} \right)\quad \textnormal{is  $a.s.$ a convergent series.}
\end{equation*}
Note that we can assume without loss of generality that $\tilde{n}_0=0$. Then, let us remark that 
\begin{align}\label{eq:diffofmartingales}
\sum_{k=1}^n \gamma_k \sum_{\ell=1}^{k} \bar{\varepsilon}^{k-\ell} \Delta M_{\ell+1}&=\sum_{\ell=1}^{n} \Delta M_{\ell+1}\big(\sum_{k=\ell}^n\gamma_k \bar{\varepsilon}^{k-\ell}\big),\nonumber\\
&=\underbrace{\sum_{\ell=1}^{n} \left(\sum_{k=\ell}^{+\infty} \gamma_k  \bar{\varepsilon}^{k-\ell}\right)\Delta M_{\ell+1}}_{\tilde{N}_n}-\left(\sum_{k=n+1}^{+\infty}\gamma_k \bar{\varepsilon}^{k}\right)\sum_{\ell=1}^{n} \bar{\varepsilon}^{-\ell}\Delta M_{\ell+1}.
\end{align}
Note that in the above equation, we provided a decomposition which allows to replace a triangular array of martingale increments by the difference of a martingale and a simpler triangular array than the previous one. For the martingale term $(\tilde{N}_n)$, recall by Step $1$ that if $\varepsilon\le \varepsilon_0$, $\sup_{n\ge1}\ES[V_\varepsilon(X_n)]<+\infty$ so that for every $\varepsilon\in(0,\varepsilon_0]$, 
$$\sup_{n\ge 1}\ES[(\Delta M_{n})^2]\le \sup_{n\ge 1}\ES[V_\varepsilon(X_n)]+\infty.$$
Then, since $(\gamma_n)$ is non-increasing, we obtain
$$\ES[\langle \tilde{N}\rangle_\infty]\le  C\sum_{\ell=2}^{+\infty} \left(\sum_{k=\ell}^{+\infty} \gamma_k \bar{\varepsilon}^{k-\ell}\right)^2\le \sum_{\ell\ge 2} \frac{\gamma_\ell^2}{(1-\bar{\varepsilon})^2}<+\infty,$$
which involves that $(\tilde{N}_n)$ $a.s.$ converges. Let us finally consider the second term of \eqref{eq:diffofmartingales}. Setting 
$b_n=\left(\sum_{k=n}^{+\infty}\gamma_k \varepsilon^{k}\right)^{-1}$ (which clearly goes to $\infty$) and $u_n= \frac{1}{b_n}\varepsilon^{-n}\Delta M_{n+1}$, we remark that 
$$ \sum_{\ell=2}^{n} u_\ell= \tilde{N}_n,$$
so that the convergence of the series follows from what preceeds.  Once again, by \Cref{lem:kron}, we derive that 
$$
\left(\sum_{k=n+1}^{+\infty}\gamma_k \varepsilon^{k}\right)\sum_{\ell=1}^{n} \varepsilon^{-\ell}\Delta M_{\ell+1}\xrightarrow{n\rightarrow+\infty}0,\quad a.s.
$$
As a conclusion, if $\varepsilon\le \varepsilon_0\wedge \tilde{\varepsilon}_0$ (where $\varepsilon_0$ and $\tilde{\varepsilon}_0$ are respectively defined in Step $1$ and in \eqref{secondcondeps}), the conclusion of the Proposition holds true.
\end{proof}
For the sake of completeness, we recall below the classical Kronecker Lemma used several times in the above proof.
\begin{lem}[Kronecker Lemma] \label{lem:kron}Let $(u_n)$ and $(b_n)$ be two real sequences. Assume $\sum u_n$ is convergent and that $(b_n)$ is positive with $\lim_{n\rightarrow+\infty} b_n=+\infty$. Then, $b_n^{-1}\sum_{k=1}^n b_k u_k\xrightarrow{n\rightarrow+\infty}0.$
\end{lem}

\section{Euler schemes of McKean-Vlasov SDEs: proofs and additional results}\label{sec:proof4}
 This section is related to Euler schemes of McKean-Vlasov SDEs. We begin by a concrete criterion which ensures  \ref{condhvdeux} for these dynamics (possibly with jumps) and then focus on the proofs of \Cref{prop:model1} and \Cref{prop:model2}.
 

\subsection{About \ref{condhvdeux}} 
\begin{lem}\label{lem:hvdeux} Assume that $\MM$ is a closed subset of $\ER^d$. Let $(\zeta_t)_{t\ge0}$ be a centered L\'evy process on $\ER^d$ and $h$ denote the step of the Euler scheme. Assume that $(x,\mu)\mapsto \sigma(x,\mu)$ is bounded (on $\MM\times{\cal P}(\MM)$). Then, \ref{condhvdeux} holds in the two following situations.
\begin{itemize}
\item[(i)](Standard Euler scheme)  The dynamics is given by \eqref{eq:formegeneralenoyau}, the step $h$ is small enough, $\zeta_1$ has moments of any order and some positive $\alpha$ and $C$ exist such that
$$ \langle x, b(x,\mu)\rangle \le C-\alpha |x|^2\quad\textnormal{and}\quad |b(x,\mu)|^2\le C(1+|x|^2) $$
\item[(ii)](Modified Euler scheme) The dynamics is given by \eqref{eq:formegeneralenoyautronc}, there exists $\alpha>0$ such that $\ES[|\zeta_h|^\alpha]<+\infty$ and Condition \ref{condhtronc} holds.
 \end{itemize}
\end{lem}
\begin{Rq} $\rhd$ In this discrete-time setting, we prefer to work with closed subsets which are ``friendlier'' for inf-compactness (for instance, $x\mapsto x$ is not inf-compact on ${D}=(0,+\infty)$ but is on $\bar{D}=(0,+\infty)$).\\

$\rhd$ The boundedness condition on $\sigma$ may be alleviated.\\

$\rhd$ The condition in $(i)$ is eventually satisfied when $(\zeta_t)_{t\ge0}$ is a Brownian motion (and more generally when the large jumps of the L\'evy process have moments of any order). Let us also recall that for a Lévy process, $\zeta_1$ has moments of any order is equivalent to: $\forall t\ge0$, $\zeta_t$ has moments of any order. 

\end{Rq}
\begin{proof} \emph{(i)} We set $V(x)=|x|^{2}$ 
 and assume that $Y_0=x$. By the second order Taylor formula and the assumptions on the coefficients, we have
$$V(Y_1)1_{\tau>1}\le V(x)(1-2\alpha h)+2 \langle x, \zeta_h\rangle+c\left(h^2 V(x)+|\zeta_h|^2\right),$$
where $c$ is independent of $h$. By the elementary inequality $|\langle u, v\rangle |\le \frac{1}{2}(\rho |u|^2+\rho^{-1} |v|^2)$ for $u,v\in\ER^d$ and $\rho>0$,
$$2 \langle x, \zeta_h\rangle\le \frac{\alpha h}{2} V(x)+\frac{2}{\alpha h} |\zeta_h|^2.$$
Hence for $h\le h_0:=\frac{\alpha}{2c}$, a finite $C_h$ exists such that
$$V(Y_1)1_{\tau>1}\le V(x)(1-{\alpha} h)+C_h(1+|\zeta_h|^2).$$
Let $\varepsilon>0$ and set 
$$ p_\varepsilon=\inf\left\{p\ge 1, (1-{\alpha} h)^p\le \frac{\varepsilon}{2}\right\}.$$
By \Cref{lem:inegelem} applied with $\delta=1$, it follows that
$$V^{p_\varepsilon}(Y_1)1_{\tau>1}\le \varepsilon V^{p_\varepsilon}(x)+C_{h,p}(1+|\zeta_h|^{2p_\varepsilon}).$$
The result follows.

\noindent \emph{(ii)} Set $V(x)=|x|^\alpha$. We have 
\begin{align*}
 \ES_x[V(Y_1) 1_{\tau>1}]\le |\tronc(x+h b(x,\mu))|^{\alpha}+\ES[|\xi_h|^{\alpha}1_{\tau>1}]
 \le  C<+\infty,
 \end{align*}
 Thus,
 $$\sup_{x\in\MM} K_{\mu,\partial} V(x)\le C,$$
 which in turn trivially implies \ref{condhvdeux}.
  \end{proof}
  \begin{lem}\label{lem:inegelem} For every $p\ge1$, for every $\delta>0$, there exists a finite constant $C_{p,\delta}$ such that for any $u,v\ge0$,
  $$(u+v)^p\le (1+\delta) u^p+ C_{p,\delta} v^p.$$
  \end{lem}
  \begin{proof}
  By the first order Taylor formula applied to $x\mapsto x^p$ on $\ER_+$,
  $$(u+v)^p= u^p+p(u+\theta v)^{p-1} v\le u^p+ p 2^{p-1} u^{p-1} v+p2^{p-1} v^p.$$
  Then, by the Young inequality  applied with $\bar{p}= p $ and $\bar{q}= \frac{p}{p-1}$
  $$ v u^{p-1} =(\rho^{1-p} v) (\rho u)^{p-1}\le \frac{\rho^{p-p^2}}{p} v^p+ \frac{(p-1)\rho^p}{p} u^p,\quad \rho>0.$$
  Taking $\rho$ such that $\frac{(p-1)\rho^p}{p}=\delta$, the result follows.

  \end{proof}

\subsection{Proof of \Cref{prop:model1}$(i)$}\label{proof:model1}

\begin{proof}
We apply \Cref{theo:discret}. Since $\bar{D}$ is compact,  $(\mu_n^h)_{n\ge1}$ is certainly almost surely tight  and we thus only have to check Assumptions $\mathbf{(A_{ i})}$, $i=0,\ldots,4$. In order to alleviate the notations, we write $K_{\mu,\partial}$ for $K_{\mu,\partial}^{(h)}$ and $g$ instead of $g_h$. We have
\begin{align}
K_{\mu,\partial} f(x)&=\ES[f(x+h b(x,\mu)+\sigma(x,\mu) \zeta_h)1_{\{x+h b(x,\mu)+\sigma(x,\mu) \zeta_h\in \bar{D}\}}]\nonumber\\
&=\int f(u) {\bf 1}_{\{u\in \bar{D}\}}g\left(\varphi_{x,\mu}(u)\right)|{\rm det} (\sigma(x,\mu))|^{-1}\lambda_d(du),\label{eq:29090}
\end{align}
where $\varphi_{x,\mu}(u)=\sigma(x,\mu)^{-1}(u-x-hb(x,\mu))$.  Set 
\begin{equation}\label{eq:compactKK}
K=\overline{\bigcup_{(x,\mu)\in\bar{D}\times{\cal P}(\bar{D})} \varphi_{x,\mu}(\bar{D})}.
\end{equation}
By $\mathbf{(H_{MV}})$, one can check that the function $(x,\mu)\mapsto \sup_{u\in \bar{D}} |\varphi_{x,\mu}(u)|$ is bounded so that the set $K$ is a compact subset of $\ER^d$.
It follows that
\begin{equation}\label{eq:minomajoexample}
 c_1 \lambda_d(D)\Psi(.)\le  K_{\mu,\partial}(x,.)\le c_2 \lambda_d(D) \Psi(.),
 \end{equation}
with 
{$$ c_1=\inf_{z\in K} g(z)\inf_{x,\mu}|\det(\sigma(x,\mu))|^{-1},\quad  c_2=\sup_{z\in K} g(z)\sup_{x,\mu}|\det(\sigma(x,\mu))|^{-1},$$
and $\Psi$ denotes the uniform distribution on $D$. Under the assumptions, $g$ is continuous and positive on $\ER^d$, and $(x,\mu)\mapsto|\det (\sigma(x,\mu))|^{-1}$ is bounded (uniform ellipticity uniformly in the measure argument)  so that $c_1$ and $c_2$ are positive finite constants.} Thus, \eqref{eq:minomajoexample} corresponds to the assumption of \Cref{lem:lowerupper}, which thus involves \ref{condhquatre}. The assumptions \ref{condhzero} and \ref{condhtrois} are also clearly true under \eqref{eq:minomajoexample}. For \ref{condhun}, the Feller property is obvious. Since \ref{condhdeux} implies the second part of \ref{condhun}, we now only need to prove \ref{condhdeux}:\\

Set $\mu_t=(1-t)\mu_0+t\mu_1$ and assume that $\zeta_h$ has density $g$. For every bounded measurable $f:\MM\rightarrow\ER$,
\begin{align}
K_{\mu_1,\partial} f(x)-&K_{\mu_0,\partial} f(x)= \int_0^1\partial_t (K_{\mu_t,\partial} f(x)) dt\nonumber\\
&=\int_0^1 \int f(u) 1_{\{u\in D\}} \partial_t\left(g(\sigma_t^{-1}(u-m_t)) |{\rm det}(\sigma_t)|^{-1} \right) \label{eq:gsifm}
\lambda_d(du) dt
\end{align}
where $\sigma_t=\sigma(x,\mu_t)$ and $m_t=x+ h b(x,\mu_t)$.
By \Cref{lem:controldensity} and the fact that $K$ defined in \eqref{eq:compactKK} is compact (so that $g$ and $\nabla g$ are bounded on $K$), we deduce that a finite constant $C$ exists such that for all $x$ and $f$,
\begin{align*}
|K_{\mu_1,\partial} f(x)-K_{\mu_0,\partial} f(x)|\le C\|f\|_\infty\|\mu_1-\mu_0\|_{TV}, 
\end{align*}
which in turn implies \ref{condhdeux}. By \Cref{theo:discret}, we can thus conclude that every weak limit $\mu_h^\star$ is a QSD for $(K_{\mu,\partial}^{(h)})_{\mu\in{\cal P}(\bar{D})}$.

\noindent {By the very definition of a QSD,  for every bounded measurable function $f:\bar{D}\rightarrow\ER$, 
\begin{equation}\label{eq:defQSD}
    \int {K}_{\mu,\partial} f(x)\mu^{\star}_{h}(dx)=\rho_{h}\langle\mu^{\star}_{h},f\rangle.
\end{equation}
with $\rho_{h}=1-\mu^{\star}_{h}(\delta_{\mu^{\star}_{h}})\in(0,1)$ (the extinction rate). Thus, for $f={\bf 1}_{\partial D}$,  this yields:
$$\mu^{\star}_{h}(\partial D)=\rho_h^{-1}  \int{K}_{\mu,\partial}(x,\partial D)\mu^{\star}_{h}(dx)=0,$$
since ${K}_{\mu,\partial}(x,\partial D)=0$ for all $x\in\bar{D}$ (this point is obvious since ${K}_{\mu,\partial}(x,.)$ has a density with respect to Lebesgue measure under the assumptions).}
\end{proof}
\begin{lem}\label{lem:controldensity} Assume \ref{condhmv}. Assume that $g$ is ${\cal C}^1$. Then, 
$$|\partial_t\left(g(\sigma_t^{-1}(u-m_t)) |{\rm det}(\sigma_t)|^{-1}\right)|\le C \|\mu_{1}-\mu_{0}\|_{TV}\left(
|\nabla g(\sigma_t^{-1}(u-m_t)|+g(\sigma_t^{-1}(u-m_t)|)\right).$$
\end{lem}
\begin{proof}
\begin{align*}
\partial_t\left(g(\sigma_t^{-1}(u-m_t))\right)&=\nabla g(\sigma_t^{-1}(u-m_t)). \partial_t(\sigma_t^{-1}(u-m_t))\\
&=-\nabla g(\sigma_t^{-1}(u-m_t))\left[\sigma_t^{-2}(\partial_t\sigma_t)(u-m_t)+\sigma_t^{-1}(\partial_t m_t)
\right].
\end{align*}
Now, by definition of a flat derivative, 
\begin{align}
    \partial_t\sigma_t=\partial_{t}\sigma(x,\mu_{t})&=\int\frac{\delta\sigma(x,\cdot)}{\delta m}(\mu_{t},y)(\mu_{1}-\mu_{0})(dy)\quad\textnormal{and,}\nonumber\\
    h^{-1}\partial_t m_t=\partial_{t}b(x,\mu_t)&=\int\frac{\delta b(x,\cdot)}{\delta m}(\mu_{t},y)(\mu_{1}-\mu_{0})(dy).\label{def:mt}
\end{align}
By Assumption \ref{condhmv}, it follows that a positive constant $C$ (independent of $h$) exists such that:
\begin{align}\label{control:sigmat}
    \|\partial_t\sigma_t\|_\infty+h^{-1}\|\partial_t m_t\|_\infty\le C \|\mu_{1}-\mu_{0}\|_{TV}.
\end{align}
Under the uniform ellipticity assumption, we get:
$$
||{\rm det}(\sigma_t)|^{-1} \partial_t\left(g(\sigma_t^{-1}(u-m_t))\right)|\le C |\nabla g(\sigma_t^{-1}(u-m_t)|\|\mu_1-\mu_0\|_{TV}.
$$
On the other hand, 
$$\partial_t ({\rm det}(\sigma_t))= {\rm Tr}(^t {\rm com}(\sigma_t)(\partial_t \sigma_t)).$$
Thus, using \eqref{control:sigmat}, the continuity of $A\mapsto {\rm ^t com}(A)$ and, the boundedness and the uniform ellipticity of $\sigma$ (under \ref{condhmv}),  we obtain:
$$ \| \partial_t ({\rm det}(\sigma_t)^{-1})\|_\infty\le C \|\mu_{1}-\mu_{0}\|_{TV}.$$
The lemma follows.
\end{proof}

\subsection{Proof of \Cref{prop:model2}}
First, we prove that $(\mu_n)_{n\ge1}$ is tight. By \Cref{lem:hvdeux}$(ii)$, Assumption \ref{condhvdeux} is satisfied. We now check that $\|\delta_\mu\|_\infty<1$.  As in the proof of \Cref{prop:model1}, we write $\tilde{K}_{\mu,\partial}$ and $g$ instead of $\tilde{K}_{\mu,\partial}^{(h)}$ and $g_h$. Similarly as in \eqref{eq:29090}, we get
\begin{equation}
\tilde{K}_{\mu,\partial} (x,du)= {\bf 1}_{\{u\in {D}\}}g\left(\tilde{\varphi}_{x,\mu}(u)\right)|{\rm det} (\sigma(x,\mu)|^{-1})\lambda_d(du),\label{eq:29090bis}
\end{equation}
with $ \tilde{\varphi}_{x,\mu}(u)=\sigma(x,\mu)^{-1}\left[u-\tronc(x+hb(x,\mu))\right].$ Let $B(a,\rho)\subset D$. The fact that $\tronc$ is bounded combined with the uniform ellipticity assumption implies that the function $(x,\mu)\mapsto \sup_{u\in B(a,\rho)} |\tilde{\varphi}_{x,\mu}(u)|$ is bounded so that the set $\tilde{K}$ defined as the adherence of ${\bigcup_{(x,\mu)\in D\times{\cal P}(D)} \varphi_{x,\mu}(B(a,\rho))}$ is compact. Then, setting $\underline{c}=\inf_{u\in \tilde{K}} g(u)$,
$$\tilde{K}_{\mu,\partial}{\bf 1}(x)\ge \underline{c}\lambda_d(B(a,r))\;\Longrightarrow\; \|\delta_\mu\|_\infty<1.$$
We then deduce tightness from \Cref{prop:tightnessbis} (and the fact that $\sum \gamma_n^2<+\infty$). \\

\noindent Now, one checks the other assumptions. For the same reasons as in the proof of \Cref{prop:model1}$(i)$ (see \Cref{proof:model1}), we only check \ref{condhdeux} and \ref{condhquatre}. For \ref{condhdeux}, the proof follows the lines of the one of \Cref{prop:model1}$(i)$ replacing $m_t$ by $\tilde{m}_t=\tronc(x+hb(x,\mu_t))$. Since $\tronc$ has a bounded Jacobian, \eqref{def:mt} and \ref{condhmv} imply that $\| \partial_t \tilde{m}_t\|_\infty<+\infty$ and one can then check that the conclusions of \Cref{lem:controldensity} are still true (replacing $m_t$ by $\tilde{m}_t$). As a consequence, since $g$ and $\nabla g$ are bounded, one thus deduces that
$$\left|\int_0^1 \int_D f(u)  \partial_t\left(g(\sigma_t^{-1}(u-\tilde{m}_t)) |{\rm det}(\sigma_t)|^{-1} \right)\lambda_d(du) dt\right|\le C \|f\|_\infty \|\mu_1-\mu_0\|_{TV},$$
which in turn implies \ref{condhdeux} (by \eqref{eq:gsifm}).\\

\noindent Now, let us consider \ref{condhquatre}. Using that $\sigma$ is uniformly elliptic (uniformly in $\mu$)  and applying \eqref{eq:conddensminomajo} with 
$m=\tronc(x+hb(x,\mu)$, $\sigma=\sigma(x,\mu)$ (and $M=\sup_{x,\mu} |\tronc(x+hb(x,\mu)|$), we deduce that there exist some positive  $c_1$ and $c_2$, and a density $\mathfrak{p}_h$ such that
$$c_1\mathfrak{p}_h(u) \le g\left(\tilde{\varphi}_{x,\mu}(u)\right)({\rm det} (\sigma(x,\mu)^{-1}))\le c_2\mathfrak{p}_h(u),$$
where $\mathfrak{p}_h$ is a density on $D$. Plugging this inequality into \eqref{eq:29090bis}, this implies \ref{condhquatre} and allows to apply \Cref{theo:discret} to conclude. The fact that $\mu^\star_h(\partial D)=0$ follows from the same arguments as for \Cref{prop:model1}$(i)$.


\subsection{Proof of \Cref{prop:model1}$(ii)$}
{This proof is an adaptation of \cite[Theorem 3.9]{BCP} for McKean-Vlasov dynamics.\\ 

Let $(W_t)_{t\ge0}$ denote a Brownian motion. For a given step $h>0$, we denote by $(\xi^{h,\mu}_{t})_{t\geq0}$, the continuous-time Euler scheme defined by $\xi^{h,\mu}_{0}=y\in D$,
\begin{align*}
    \forall n\in\mathbb{N},\quad\xi^{h,\mu}_{(n+1)h}&:=\xi^{h,\mu}_{nh}+b(\xi^{h,\mu}_{nh},\mu)h+\sigma(\xi^{h,\mu}_{nh},\mu)(W_{(n+1)h}-W_{nh}),
\end{align*}
and, 
\begin{align*}
    \forall t\in[nh,(n+1)h),\quad\xi^{h,\mu}_{t}=\xi^{h,\mu}_{nh}+b(\xi^{h,\mu}_{nh},\mu)(t-nh)+\sigma(\xi^{h,\mu}_{nh},\mu)(W_{t}-W_{nh}).
\end{align*}
Note that the continuous Euler scheme is an It\^o process verifying
\begin{align}\label{eq:contEusch}
    \xi^{h,\mu}_{t}&=\xi^{h,\mu}_{0}+\int_{0}^{t}b(\xi^{h,\mu}_{\underline{s}_h},\mu)ds+\int_{0}^{t}\sigma(\xi^{h,\mu}_{\underline{s}_h},\mu)dW_{s},\quad \textnormal{with}\quad \underline{s}_h= h\lfloor \frac{s}{h}\rfloor.
    \end{align}
With the notations of \eqref{eq:kmupartialh}, the kernel of the killed process related to the Markov chain $(\xi_{nh}^{h,\mu})_{n\ge0}$ is still denoted by 
$K_{\mu,\partial}^{(h)}$ but keeping in mind that $\zeta=W$ is now a Brownian motion.  \\

Before going further, we first deduce from \eqref{eq:defQSD} and from an induction that for all  $t\ge0$, for every bounded measurable function $f:\mathbb{R}^{d}\to\mathbb{R}$,
\begin{equation}\label{eq:QSDrelationshipt}
    \mathbb{E}_{\mu^{\star}_{h}}\left[f(\xi^{h,\mu^{\star}_{h}}_{\underline{t}_h})1_{\{\bar{\tau}^h_{D}(\xi^{h,\mu^{\star}_{h}})>t\}}\right]=e^{-\lambda_{h}\underline{t}_h}\mu^{\star}_{h}(f),
\end{equation}
with $\lambda_{h}:=\frac{\log \rho_{h}}{h}$ and for a process $(Z_t)_{t\ge0}$,
\begin{equation}\label{discretestop}
\bar{\tau}^h_{D}(Z)=\inf\{nh, Z_{nh}\in D^c\}.
\end{equation}
In other words, $\bar{\tau}^h_{D}$ denotes the exit time related to the discrete-time dynamics (this notation will be still used in the sequel). Note that in order to write \eqref{eq:QSDrelationshipt}, we used that $\{\bar{\tau}^h_{D}(\xi^{h,\mu^{\star}_{h}})>t\}=\{\bar{\tau}^h_{D}(\xi^{h,\mu^{\star}_{h}})>\un{t}_h\}$.

 To prove the result, we now need  to make $h$ go to $0$ in \eqref{eq:QSDrelationshipt}. \\
\noindent \textit{Step 1} (Tightness of the sequence of QSDs):
{We show that $(\mu^{\star}_{h})_{h\leq h_{0}}$ is tight on $D$. For $\delta>0$, set
\begin{align}\label{eq:bdelta}
    B_{\delta}:=\bigg\{x\in D,\quad d(x,\partial D)\leq\delta\bigg\}.
\end{align}
It is enough to prove that for every $\varepsilon>0$, there exists $\delta_{\varepsilon}>0$ such that for every $h\in(0,h_{0}]$, $\mu^{\star}_{h}(B_{\delta_{\varepsilon}})\leq\varepsilon$. \\

From \eqref{eq:QSDrelationshipt} applied with $f={\bf 1}_{B_\delta}$, we get in particular
$$ \mu^{\star}_{h}(B_\delta)\le \PE_{\mu_h^\star}(\xi^{h,\mu^{\star}_h}_{\underline{t}_h}\in B_\delta)\le \lambda (B_\delta)\sup_{(x,y)\in\bar{D}^2} p_{\underline{t}_h}^{h,\mu^\star_h}(x,y),$$
where $p_t^{h,\mu}(x,.)$ denotes the density of the Euler scheme at time $t$ (which exists under the uniform ellipticity assumption). To bound the density uniformly in $h$ and $\mu\in{\cal P}(D)$, we use \cite[Theorem 2.1]{LemaireMenozzi_2010} (which provides Aronson estimates for the density). One checks that the related assumptions $\mathbf{(UE)}$ and $\mathbf{(SB)}$ are satisfied with some parameters $L_0$ and $\lambda_0$ independent of $\mu$ under our assumption \ref{condhmv}. Thus, by this theorem, we deduce that for all $T_0<T_1$, a constant $C$ exists such that for all $h\in(0,h_0)$, for all $t\in[T_0,T_1]$,
\begin{equation}\label{eq:densityboundeuler}
\sup_{(x,y)\in\bar{D}^2} p_{t}^{h,\mu^\star_h}(x,y)\le C.
\end{equation}
The tightness follows.\\

\noindent From now, we can thus consider a weak limit $\tilde{\mu}$ of a convergent subsequence $(\mu^{\star}_{h_n})_{n\ge1}$ with $h_n\rightarrow0$.\\

\noindent \textit{Step 2} (About the left-hand side of \eqref{eq:QSDrelationshipt}): we  show that for all $t\ge0$ and for all bounded Lipschitz continuous function $f:D\mapsto \ER$,

\begin{equation}\label{eq:emustarhn} \mathbb{E}_{\mu^{\star}_{h_n}}\left[f(\xi^{h_n,\mu^{\star}_{h_n}}_{\underline{t}_{h_n}})1_{\bar{\tau}^{h_n}_{D}(\xi^{h_n,\mu^{\star}_{h_n}})>\underline{t}_{h_n}}\right]\xrightarrow{n\rightarrow+\infty}\mathbb{E}_{\mu^{\star}}\left[f(\xi^{\mu^{\star}}_{t})1_{{\tau}_{D}(\xi^{\mu^{\star}})>t}\right],
\end{equation}
where $\xi^{\mu^\star}$ satisfies:
 $$\xi^{\mu^{\star}}_t=\xi^{\mu^{\star}}_0+\int_0^t b(\xi^{\mu^{\star}}_s,\mu^\star) ds+\int_0^{t} \sigma(\xi^{\mu^{\star}}_s,\mu^\star) dW_s.$$
In order to simplify the notations, we write $h$ instead of $h_n$ in the sequel of this step and assume that $\tilde{\mu}^\star_h\rightarrow\mu^\star$ when $h\rightarrow0$. This does not change the proof. We first decompose the error related to \eqref{eq:emustarhn} into three parts:
\begin{align}
\mathbb{E}_{\mu^{\star}_{h}}\Big[f(\xi^{h,\mu^{\star}_{h}}_{t})1_{\bar{\tau}^{h}_{D}(\xi^{h,\mu^{\star}_{h}})>t}\Big]-\mathbb{E}_{\mu^{\star}}\Big[f(\xi^{\mu^{\star}}_{t})1_{{\tau}_{D}(\xi^{\mu^{\star}})>t}&\Big]=\mu^\star_h(\bar{\cal P}_{t}^{h,\mu^\star_h,\bar{\tau}}f-\bar{\cal P}_{t}^{h,\mu^\star,\bar{\tau}}f)\label{eq:decompa41}\\
&+\mu^\star_h(\bar{\cal P}_{t}^{h,\mu^\star,\bar{\tau}}f)-\mu^\star_h(P_t^{\mu^\star,\partial} f)\label{eq:decompa42}\\
&+\mu^\star_h(P_t^{\mu^\star,\partial}f)-\mu^\star(P_t^{\mu^\star,\partial}f)\label{eq:decompa43},
%
\end{align}
where,
\begin{align*}
\bar{\cal P}_{t}^{h,\mu,\bar{\tau}}f(x)=\ES_x\left[f(\xi^{h,\mu}_{t})1_{\bar{\tau}_{D}^h(\xi^{h,\mu})>{t}}\right]\quad\textnormal{and,}\quad
P_{t}^{\mu,\partial} f(x)=\ES_x[f(\xi^{\mu}_t){\bf 1}_{\tau_D(\xi^\mu)>t}].
\end{align*}
To control the right-hand side of \eqref{eq:decompa41}, we first use that $f$ is bounded Lipschitz continuous to obtain that for all $T>0$, for all $t\in[0,T]$:
\begin{align}
\nonumber |\bar{\cal P}_{t}^{h,\mu^\star_h,\bar{\tau}}f(x)-\bar{\cal P}_{t}^{h,\mu^\star,\bar{\tau}}f(x)|&\le [f]_1\ES_x[|\xi^{h,\mu^\star_h}_{t}-\xi^{h,\mu^\star}_{t}|]\\
&+\|f\|_\infty \PE(\{\bar{\tau}_{D}^h(\xi^{h,\mu^{\star}_h})>{t}\}\Delta \{\bar{\tau}_{D}^h(\xi^{h,\mu^\star})>{t}\})\\
\label{eq:otherpartsum}\le C_T \Big({\cal W}_1&(\mu^{\star}_h,\mu^\star)+   \PE(\{\bar{\tau}_{D}^h(\xi^{h,\mu^{\star}_h})>{t}\}\Delta \{\bar{\tau}_{D}^h(\xi^{h,\mu^\star})>{t}\})\Big),
\end{align}
where in the second line, we used \Cref{lem:muomu1} with $x=y$, $\mu_0=\mu^{\star}_h$ and $\mu_1=\mu^{\star}$. Since in this bounded setting, the weak convergence implies the ${\cal W}_1$-convergence, we only focus on the second term. Let us first consider the first part of the symmetric difference:
\begin{equation}\label{eq:taubardsyl}
\PE(\bar{\tau}_{D}^h(\xi^{h,\mu^{\star}_h})>{t}, \bar{\tau}_{D}^h(\xi^{h,\mu^\star})<{t}).
\end{equation}
To manage this term, we apply \Cref{lem:xyexit} with $X=\xi^{h,\mu^{\star}_h}$ and $Y=\xi^{h,\mu^{\star}}$. By \Cref{lem:muomu1}, $(iii)$ is satisfied with $c_3=C_T \delta^{-1}  {\cal W}_1(\mu^{\star}_h,\mu^\star)$. For $(ii)$, we use \cite[Theorem 2.4]{BGGobet} which states that for a given $\mu$,
\begin{equation}\label{eq:taubadrre}
\ES_x[\bar{\tau}_{D}^h(\xi^{h,\mu})]\le d_{(2.4)}[\tilde{\psi}_D(x)+\sqrt{h}],
\end{equation}
where $\tilde{\psi}_D(x)$ is a smooth extension of the algebraic distance to the boundary. More precisely, by \cite[Lemma 2.4]{panloup_reygner}, it can be shown that $P_t:=\tilde{\psi}_D(\xi^{h,\mu}_t)$ satisfies the assumption $\mathbf{(P)}$ of \cite{BGGobet} (see also Lemma A.3 of \cite{BGGobet} for details). By Lemma A.1 and A.4 of \cite{BGGobet} , one can check that the other assumptions of  \cite[Theorem 2.4]{BGGobet} hold true under our assumptions and that the constant denoted by $d_{(2.4)}$ is uniform in $\mu\in{\cal P}(D)$, $t\le T$ and $h\le h_0$. We thus deduce from \eqref{eq:taubadrre} and Markov inequality that Assumption $(ii)$ of  \Cref{lem:xyexit} holds true with $c_2= C(\delta+\sqrt{h})$ where $C$ is a constant independent of the parameters. \\

\noindent Finally, let us consider Assumption $(i)$ of  \Cref{lem:xyexit} with $\mu=\mu^\star_h$. Since $(\mu^\star_h)$ is tight, one can find for all $\varepsilon>0$, a $\delta_{\varepsilon}>0$ such that $B_{\delta_\varepsilon}$ defined by \eqref{eq:bdelta} satisfies: for all $h\in(0,h_0]$, $\mu_h^{\star}(B_{\delta_\varepsilon})$. 
Thus, 
\begin{align}
\PE_{\mu^\star_h}(\bar{\tau}_{D}^h(\xi^{h,\mu^{\star}_h})\in [t-t_0,t])
&\le \varepsilon+\sup_{x, d(x,\partial D)\ge \delta_{\varepsilon}}\PE_x((\bar{\tau}_{D}^h(\xi^{h,\mu^{\star}_h})\le 2 t_0)\nonumber\\
&+ \sup_{x, d(x,\partial D)\ge \delta_{\varepsilon}, t\ge 2 t_0}\PE_x(\bar{\tau}_{D}^h(\xi^{h,\mu^{\star}_h})\in [(t-t_0),t]).\label{eq:2039301}
\end{align}
We control the two right-hand terms. First, for all $x \in D$ such that $d(x,\partial D)\ge \delta_{\varepsilon}$,
\begin{align*}
 \PE_x((\bar{\tau}_{D}^h(\xi^{h,\mu^{\star}_h})&\le 2 t_0)\le \PE_x(\sup_{t\le 2t_0} |\bar{\tau}_{D}^h(\xi^{h,\mu^{\star}_h}_t-x|\ge \delta_{\varepsilon})\le \delta_{\varepsilon}^{-2} \ES_x
[\sup_{t\le 2t_0} |\bar{\tau}_{D}^h(\xi^{h,\mu^{\star}_h}_t-x|^2]\\
&\le 2 \delta_{\varepsilon}^{-2} \left(\ES[\int_0^{2t_0} |b(\xi^{h,\mu^{\star}_h}_{\un{t}_h},\mu^\star_h) dt|^2]+\ES[\left|\sup_{t\le 2t_0} \int_0^t \sigma(\xi^{h,\mu^{\star}_h}_{\un{t}_h},\mu^\star_h) dW_s\right|^2]\right)\\
&\le C \delta_{\varepsilon}^{-2} (t_0^2+t_0)\le C \delta_{\varepsilon}^{-2} t_0.
\end{align*}
where the constant $C$ is universal since the coefficients $b$ and $\sigma$ are uniformly bounded in $x$ and $\mu$ (since they are uniformly Lispschitz on a compact set).
For \eqref{eq:2039301}, we decompose into two parts and deduce from the strong Markov property that
$$ \PE_x(\bar{\tau}_{D}^h(\xi^{h,\mu^{\star}_h})\in [(t-t_0),t])\le 
\PE(\xi^{h,\mu^{\star}_h}_{t-t_0}\in B_{\delta_{\varepsilon}})+\sup_{x, d(x,\partial D)\ge \delta_{\varepsilon}}\PE_x((\bar{\tau}_{D}^h(\xi^{h,\mu^{\star}_h})\le 2 t_0).
$$
The second term can be managed as previously whereas the first one satisfies the inequality:
$$\PE(\xi^{h,\mu^{\star}_h}_{t-t_0}\in B_{\delta_{\varepsilon}})\le C \lambda(B_{\delta_{\varepsilon}}),$$
owing to the inequality \eqref{eq:densityboundeuler}. This means that we can set
$$c_3= \varepsilon+ C\lambda(B_{\delta_{\varepsilon}})+ \delta_{\varepsilon}^{-2} t_0,$$
where $C$ only depends on $T$. Applying \Cref{lem:xyexit}, we get for all $T>0$, there exists a constant $C$ such that for all $t\in(0,T]$,  for all positive $\varepsilon$ and $\delta$,
\begin{align*}
 \PE(\bar{\tau}_{D}^h(\xi^{h,\mu^{\star}_h})>{t},& \bar{\tau}_{D}^h(\xi^{h,\mu^\star})<{t})\\
 &\le C \left[\delta^{-1}  {\cal W}_1(\mu^{\star}_h,\mu^\star)+
\delta+\sqrt{h}+\varepsilon+ \lambda(B_{\delta_{\varepsilon}})+ \delta_{\varepsilon}^{-2} t_0\right],
\end{align*}
where $\delta_\varepsilon$ is a positive number which can be taken as small as we want (for a given $\varepsilon$). Taking $\delta$ small enough and then small enough $h$, $\varepsilon$, $\delta_\varepsilon$ and finally $t_0$, we thus obtain the expected control of \eqref{eq:taubardsyl}. With the same arguments, one controls the other part of the symmetric difference of \eqref{eq:otherpartsum} and this ends the study of \eqref{eq:decompa41}.\\

\noindent Now, for \eqref{eq:decompa42}, this corresponds to the convergence of the killed semi-group of the Euler scheme towards the killed semi-group of the related diffusion with coefficients $b(.,\mu^\star)$ and $\sigma(.,\mu^\star)$. This can be obtained by a similar strategy as before that is left to the reader.\\

\noindent Finally, for  \eqref{eq:decompa43}, we need to prove that the map $x\mapsto P_t^{\mu^\star,\partial} f(x)$ is continuous on $D$. Denote by $\xi^x$ the process solution starting from $x$ to the diffusion, $d\xi_t=b(\xi_t,\mu^\star) dt+\sigma(\xi_t,\mu^\star) dW_t$. Under our assumptions, 
$\sup_{(x,y)\in D^2} \ES[\sup_{t\le T} |\xi_t^x-\xi_t^y|^2]\le C_T|x-y|^2$ so that, keeping in mind that $f$ is Lipschitz continuous, it is thus enough to control $\PE(\{\tau_D(\xi^x)>t\}\Delta \{\tau_D(\xi^y)<t\})$.  This can be again obtained with an adaptation of \Cref{lem:xyexit} where for a given $x_0\in D$, one sets $\delta=\frac{1}{2} d(x_0,\partial D)$. 
This concludes the proof of this step.\\

\noindent \textit{Step 3.} (Bounds for the extinction rate) We prove that some positive $\lambda_{\min}$, $\lambda_{\max}$ and $h_{0}$ exist such that for any $h\in(0,h_{0})$, we have 
\begin{align*}
    \lambda_{\min}\le \lambda_{h}\le \lambda_{\max}.
\end{align*}}
For the lower-bound, it is enough to check that for a given $t_0>0$
$$\sup_{(\mu,x,h)\in{\cal P}(D)\times D\times (0,h_0]} \PE_x(\bar{\tau}_D(\xi^{h,\mu})>\lfloor t_0 \rfloor_h)<1.$$
For this point, we can for instance again use \cite[Theorem 2.1]{LemaireMenozzi_2010} which provides lower-bounds on the density:  taking a given compact set $K_0$ in $D^c$, this result involves that under our uniform boundedness, Lipschitz and ellipticity assumptions, for $t_0\ge 2 h_0$, there exists a positive constant $C$ such that for every $(\mu,x,h)\in{\cal P}(D)\times D\times (0,h_0]$,
$$  \PE(\bar{\tau}_D(\xi^{h,\mu})\le \lfloor t_0 \rfloor_h)\ge \PE(\xi^{h,\mu}_{ \lfloor t_0 \rfloor_h}\in K_0)\ge C\lambda_d(K_0)>0,$$
as soon as $K_0$ non-empty interior.  
Second, one remark that $(\lambda_{h})_{h\le h_0}$ is upper-bounded as soon as there exists $t_0>0$ such that
$$\inf_{h>0} \mathbb{P}_{\mu^{\star}_{h}}\left(\{\bar{\tau}^h_{D}(\xi^{h,\mu^{\star}_{h}})>t_0\right)>0.$$
Denoting by $K_1$, a compact subset of $D$ such that $\inf_{h\in(0,h_0]}\mu^\star_h(K_1)>0$ (such a $K_1$ owing to the tightness property), one can check that the above property is true if 
$$\sup_{(\mu,x,h)\in{\cal P}(D)\times K_1\times (0,h_0]} \PE_x({\tau}_D(\xi^{h,\mu})\le  t_0)<1.$$
But this property is easy to ensure. Actually, with standard Gronwall arguments, it can be shown that $\ES_x[\sup_{0\le t\le t_0} |\xi^{h,\mu}_{t}-x|^2]\le C t$ (where $C$ is independent of the parameters under the assumptions of the result) and, setting $\delta=d(K_1,\partial D)$,  it follows that 
$$ \PE_x({\tau}_D(\xi^{h,\mu})\le t_0)\le \PE(\sup_{t\le t_0} |\xi^{h,\mu}_t-x|\ge \delta)\le C\delta^{-2} t_0.$$
For $t_0$ small enough, $C\delta^{-2} t_0<1$ and the result follows.\\

\noindent \textit{Step 4.} (Conclusion) Recall that $(\mu^\star_{h_n})_{n\ge1}$ denote a convergent subsequence of $(\mu^\star_{h})_{h\le h_0}$ to $\mu^\star$. At the price of an additional extraction, we deduce from Step 3 that $\lambda_{h_n}\rightarrow\lambda^\star\in(0,+\infty)$. Then, from Step 2 and  \eqref{eq:QSDrelationshipt}, we get for all $t\ge0$ and for all Lipschitz continuous function $f:D\rightarrow\ER$,
\begin{equation}
\mathbb{E}_{\mu^{\star}}\left[f(\xi^{\mu^{\star}}_{t}){\bf 1}_{\{{\tau}_{D}(\xi^{\mu^{\star}})>t\}}\right]=e^{-\lambda^\star{t}}\mu^{\star}(f).
\end{equation}
This concludes the proof.

\begin{lem}\label{lem:muomu1} For $\mu_0$ and $\mu_1\in {\cal P}(\bar{D})$, let $(\xi^{h,\mu_0}_t)_{t\ge0}$ and $(\xi^{h,\mu_1}_t)_{t\ge0}$ be given by \eqref{eq:contEusch} starting from initial conditions $x$ and $y$ respectively. Then, under the assumptions of \Cref{prop:model1}$(ii)$, a universal constant $C$ exists such that for all $T\ge0$,
$$ \ES\left[\sup_{t\le T} |\xi^{h,\mu_{1}}_{t}-\xi^{h,\mu_{0}}_{t}|^{2}\right]\le e^{C T } \left(|x-y|^2+ C{\cal W}_1(\mu_0,\mu_1)^2\right).$$
\end{lem}

By Itô formula, we have
\begin{equation}\label{eq:itof}
\frac{1}{2}|\xi^{h,\mu_{1}}_{t}-\xi^{h,\mu_{0}}_{t}|^{2}=
\frac{1}{2}|\xi^{h,\mu_{1}}_{{\un{t}_h}}-\xi^{h,\mu_{0}}_{{\underline{t}_h}}|^2+
\int_{\un{t}_h}^t F(\xi^{h,\mu_{1}}_{s},\xi^{h,\mu_{0}}_{s},\xi^{h,\mu_{1}}_{\underline{s}_h},\xi^{h,\mu_{0}}_{\underline{s}_h})ds+M_t,
\end{equation}
where,
\begin{align*}
 F(x,y,\un{x},\un{y})&=\langle x-y, b(\un{x},\mu_1)-b(\un{y},\mu_0)\rangle\\
 &+ \frac{1}{2} {\rm Tr}((\sigma(\un{x},\mu_1)-\sigma(\un{y},\mu_0))(\sigma(\un{x},\mu_1)-\sigma(\un{y},\mu_0))^*),
 \end{align*}
and,
$$ M_t=\int_0^t \langle \xi^{h,\mu_{1}}_{s}-\xi^{h,\mu_{0}}_{s}, (\sigma(\xi^{h,\mu_{1}}_{\underline{s}_h},\mu_1)-\sigma(\xi^{h,\mu_{0}}_{\underline{s}_h},\mu_0)) dW_s\rangle.$$
On the one hand, the definition of the flat derivative yields:
\begin{equation}\label{eq:xmoinsyb}
\begin{split}
\langle x-y, b(\un{x},\mu_1)-b(\un{y},\mu_0)\rangle&=\langle x-y, b(\un{x},\mu_1)-b(\un{y},\mu_1)\rangle\\
&+\int_{0}^{1}\int \langle x-y, \frac{\delta b_{\un{y}}}{\delta m} (\mu_t,z) \rangle(\mu_1-\mu_0)(dz) dt,
\end{split}
\end{equation}
with $b_{\un{y}}$ being the application $\mu\mapsto b(\un{y},\mu)$ and $\mu_t=(1-t)\mu_0+ t\mu_1$. Setting $f(z)= \langle x-y, \frac{\delta b_{\un{y}}}{\delta m} (\mu_t,z)\rangle$,
\begin{align*}
 \int \langle x-y, &\frac{\delta b_{\un{y}}}{\delta m} (\mu_t,z) \rangle(\mu_1-\mu_0)(dz)=\mu_1(f)-\mu_0(f)\\
 &\le [f]_1 {\cal W}_1(\mu_0,\mu_1)\le |x-y| \left(  \sup_{(z,\mu)\in\MM\times{\cal P}(\MM)}\bigg[\frac{\delta b(z,\cdot)}{\delta m}(\mu,\cdot)\bigg]_{1} \right){\cal W}_1(\mu_0,\mu_1).
 \end{align*}
Thus, using the Lipschitz assumption on the flat derivative (and the fact that $b$ is uniformly Lipschitz in the first variable for the first right-hand term of \eqref{eq:xmoinsyb}),
we get the existence of a constant $C$ (independent of the parameters) such that
$$\langle x-y, b(\un{x},\mu_1)-b(\un{y},\mu_0)\rangle\le C|x-y|\left(|\un{x}-\un{y}|+ {\cal W}_1(\mu_0,\mu_1)\right).$$
This leads to 
$$\langle x-y, b(\un{x},\mu_1)-b(\un{y},\mu_0)\rangle\le C(|x-y|^2+|\un{x}-\un{y}|^2+ {\cal W}_1(\mu_0,\mu_1)^2),$$
by using the elementary inequality $|uv|\le \frac{1}{2}(|u|^2+|v|^2)$ (and by keeping the same notation for the constant $C$ even it may have changed).
Similar arguments can be used for the diffusion part to obtain:
$${\rm Tr}((\sigma(\un{x},\mu_1)-\sigma(\un{y},\mu_0))(\sigma(\un{x},\mu_1)-\sigma(\un{y},\mu_0))^*)\le  C\left(|\un{x}-\un{y}|^2+ {\cal W}_1(\mu_0,\mu_1)^2\right).$$
Hence, plugging these controls into \eqref{eq:itof} and using that $(M_t)_{t\ge0}$ is a true martingale leads to:
$$ u_t\le u_{\un{t}_h}(1+C (t-\un{t}_h))+ C(t-\un{t}_h){\cal W}_1(\mu_0,\mu_1)^2+C\int_{\un{t}_h}^t u_s ds,$$
with $u_t= \frac{1}{2}\ES[|\xi^{h,\mu_{1}}_{t}-\xi^{h,\mu_{0}}_{t}|_{2}].$ The Gronwall lemma then leads to: for all $n\ge0$, for all $t\in[nh,(n+1)h]$,
$$ u_{t}\le \left(u_{nh}(1+C h)+ Ch{\cal W}_1(\mu_0,\mu_1)^2\right) e^{Ch}\le u_{nh} e^{Ch}+ Ch{\cal W}_1(\mu_0,\mu_1)^2.$$
Once again, we do not change the notation for the constant $C$ (even it may change from line to line). By an induction, we obtain for $h$ small enough:
$$ \ES\left[ |\xi^{h,\mu_{1}}_{t}-\xi^{h,\mu_{0}}_{t}|^{2}\right]\le e^{C t} \left(u_0+ C{\cal W}_1(\mu_0,\mu_1)^2\right).$$
We thus finally need to deduce that the supremum can be taken inside the expectation. This is a classical procedure again based on the Gronwall applied to 
$v_t=\ES\left[ \sup_{0\le s\le t} |\xi^{h,\mu_{1}}_{s}-\xi^{h,\mu_{0}}_{s}|^{2}\right]$
using the previous controls on $(u_t)$ and the nice properties on the supremum of a martingale  (see \emph{e.g.} the proof of Theorem 7.2. in \cite{pages_book} for a similar argument in a slightly different setting).

\begin{lem} \label{lem:xyexit} Let $(X_t)_{t\ge0}$ and $(Y_t)_{t\ge0}$ denote two Markov processes with respect to the same filtration $({\cal F}_t)_{t\ge0}$ with values in ${D}$. Let $\mu$ be a probability on $D$. Let $t_0$, $t$ be positive numbers such that $t_0<t$. Denote by $\tau_X$ and $\tau_Y$ the exit times of $D$ and assume that.\\

\noindent $(i)$ $\PE_\mu (\tau_X\in [t-t_0,t])\le c_1.$\\

\noindent $(ii)$ $\sup_{x, d(x,\partial D)\le \delta } \PE_x(\tau_Y>t_0)\le c_2.$\\

\noindent $(iii)$  $ \PE_\mu(\sup_{s\le t}|X_s-Y_s|\ge\delta)\le c_3$.\\

Then, 
$$\PE_\mu(\tau_Y>t, \tau_X<t)\le c_1+c_2+c_3.$$
\end{lem}
\begin{proof}  We have:
\begin{align*}
\PE_\mu(\tau_Y&>t, \tau_X<t)\le \PE_\mu(\tau_X\in [t-t_0,t])+ \PE_\mu(\tau_Y>t, \tau_X<t-t_0)\\
&\le c_1+ \PE_\mu (|X_{\tau_X}-Y_{\tau_X}|\ge \delta)+\PE(\tau_Y>t, d(Y_{\tau_X},\partial D)\le \delta, \tau_X\le t-t_0)\\
&\le c_1+ c_3+\PE(\tau_Y>t, d(Y_{\tau_X},\partial D)\le \delta, \tau_X\le t-t_0)\\
&\le c_1+c_3+ \ES[\PE(\tau_Y>t-\tau_X|{\cal F}_{\tau_X}){\bf 1}_{d(Y_{\tau_X},\partial D)\le \delta, \tau_X\le t-t_0)}],
\end{align*}
where in the last line, we used the strong Markov property. The result follows.
\end{proof}
\bibliographystyle{alpha}
\bibliography{ref}

\newcommand{\etalchar}[1]{$^{#1}$}
\begin{thebibliography}{LBLLP12}

\bibitem[BC15]{BC15}
Michel Bena{\"{\i}}m and Bertrand Cloez.
\newblock A stochastic approximation approach to quasi-stationary distributions
  on finite spaces.
\newblock {\em Electron. Commun. Probab.}, 20:no. 37, 14, 2015.

\bibitem[BCP18]{BCP}
Michel Benaim, Bertrand Cloez, and Fabien Panloup.
\newblock Stochastic approximation of quasi-stationary distributions on compact
  spaces and applications.
\newblock {\em Ann. Appl. Probab.}, 28(4):2370--2416, 2018.

\bibitem[BCV21]{benaim_champagnat_villemonais}
Michel Bena\"im, Nicolas Champagnat, and Denis Villemonais.
\newblock Stochastic approximation of quasi-stationary distributions for
  diffusion processes in a bounded domain.
\newblock {\em Ann. Inst. Henri Poincar\'e{} Probab. Stat.}, 57(2):726--739,
  2021.

\bibitem[Ben99]{B99}
Michel Bena{\"{\i}}m.
\newblock Dynamics of stochastic approximation algorithms.
\newblock In {\em S\'eminaire de {P}robabilit\'es, {XXXIII}}, volume 1709 of
  {\em Lecture Notes in Math.}, pages 1--68. Springer, Berlin, 1999.

\bibitem[BGG17]{BGGobet}
Bruno Bouchard, Stefan Geiss, and Emmanuel Gobet.
\newblock First time to exit of a continuous {I}t\^o{} process: general moment
  estimates and {$\rm L_1$}-convergence rate for discrete time approximations.
\newblock {\em Bernoulli}, 23(3):1631--1662, 2017.

\bibitem[BGZ14]{BGZ}
J.~{Blanchet}, P.~{Glynn}, and S.~{Zheng}.
\newblock {Theoretical analysis of a Stochastic Approximation approach for
  computing Quasi-Stationary distributions}.
\newblock {\em ArXiv e-prints}, January 2014.

\bibitem[BHM00]{BHM00}
Krzysztof Burdzy, Robert Ho{\l}yst, and Peter March.
\newblock A {F}leming-{V}iot particle representation of the {D}irichlet
  {L}aplacian.
\newblock {\em Comm. Math. Phys.}, 214(3):679--703, 2000.

\bibitem[BLR02]{BLR02}
Michel Bena{\"{i}}m, Michel Ledoux, and Olivier Raimond.
\newblock Self-interacting diffusions.
\newblock {\em Probab. Theory Related Fields}, 122(1):1--41, 2002.

\bibitem[CCL{\etalchar{+}}09]{cattiaux_collet}
Patrick Cattiaux, Pierre Collet, Amaury Lambert, Servet Mart\'inez, Sylvie
  M\'el\'eard, and Jaime San~Mart\'in.
\newblock Quasi-stationary distributions and diffusion models in population
  dynamics.
\newblock {\em Ann. Probab.}, 37(5):1926--1969, 2009.

\bibitem[CCM16]{chazottes_collet}
J.-R. Chazottes, P.~Collet, and S.~M\'el\'eard.
\newblock Sharp asymptotics for the quasi-stationary distribution of
  birth-and-death processes.
\newblock {\em Probab. Theory Related Fields}, 164(1-2):285--332, 2016.

\bibitem[CJM{\etalchar{+}}23]{cloez_esaim}
Bertrand Cloez, Lucas Journel, Pierre Monmarche, Boris Nectoux, and Mouad
  Ramil.
\newblock Recent advances in the long-time analysis of killed degenerate
  processes and their particle approximation.
\newblock In {\em Journ\'ees {SMAI} 2021}, volume~75 of {\em ESAIM Proc.
  Surveys}, pages 60--85. EDP Sci., Les Ulis, 2023.

\bibitem[CT13]{CT13}
Bertrand {Cloez} and Marie-No\'{e}mie {Thai}.
\newblock {Quantitative results for the Fleming-Viot particle system in
  discrete space}.
\newblock {\em ArXiv e-prints}, December 2013.

\bibitem[CV14]{CV14}
N.~{Champagnat} and D.~{Villemonais}.
\newblock {Exponential convergence to quasi-stationary distribution and
  Q-process}.
\newblock {\em ArXiv e-prints}, April 2014.
\newblock accepted for publication in Probability Theory and Related Fields.

\bibitem[DMM00]{MM00}
P.~Del~Moral and L.~Miclo.
\newblock A {M}oran particle system approximation of {F}eynman-{K}ac formulae.
\newblock {\em Stochastic Process. Appl.}, 86(2):193--216, 2000.

\bibitem[GLWZ22]{guillin_wu}
Arnaud Guillin, Wei Liu, Liming Wu, and Chaoen Zhang.
\newblock Uniform {P}oincar\'e{} and logarithmic {S}obolev inequalities for
  mean field particle systems.
\newblock {\em Ann. Appl. Probab.}, 32(3):1590--1614, 2022.

\bibitem[LBLLP12]{lebris_lelievre}
Claude Le~Bris, Tony Leli\`evre, Mitchell Luskin, and Danny Perez.
\newblock A mathematical formalization of the parallel replica dynamics.
\newblock {\em Monte Carlo Methods Appl.}, 18(2):119--146, 2012.

\bibitem[Lem07]{lemaire}
Vincent Lemaire.
\newblock An adaptive scheme for the approximation of dissipative systems.
\newblock {\em Stochastic Process. Appl.}, 117(10):1491--1518, 2007.

\bibitem[ML10]{LemaireMenozzi_2010}
Stéphane Menozzi and Vincent Lemaire.
\newblock On some non asymptotic bounds for the euler scheme.
\newblock {\em Electronic Journal of Probability}, 15(none), January 2010.

\bibitem[Pag18]{pages_book}
Gilles Pag\`es.
\newblock {\em Numerical probability}.
\newblock Universitext. Springer, Cham, 2018.
\newblock An introduction with applications to finance.

\bibitem[PR23]{panloup_reygner}
Fabien Panloup and Julien Reygner.
\newblock Asymptotically unbiased approximation of the qsd of diffusion
  processes with a decreasing time step euler scheme, 2023.

\bibitem[TN22]{tough_nolen}
Oliver Tough and James Nolen.
\newblock The {F}leming-{V}iot process with {M}c{K}ean-{V}lasov dynamics.
\newblock {\em Electron. J. Probab.}, 27:Paper No. 101, 72, 2022.

\bibitem[Vil14]{V11}
Denis Villemonais.
\newblock General approximation method for the distribution of {M}arkov
  processes conditioned not to be killed.
\newblock {\em ESAIM Probab. Stat.}, 18:441--467, 2014.

\bibitem[WRS20]{wang_roberts}
Andi~Q. Wang, Gareth~O. Roberts, and David Steinsaltz.
\newblock An approximation scheme for quasi-stationary distributions of killed
  diffusions.
\newblock {\em Stochastic Process. Appl.}, 130(5):3193--3219, 2020.

\end{thebibliography}

\appendix
\section{Postponed proofs}\label{sec:proofsappendix}
\begin{proof}[Proof of \cref{lemao}]
We prove that there exists $\ell\in\mathbb{N}^*$ such that
\begin{equation}\label{supkmufini}
\rho(\ell):=\sup_{(x,\mu)\in\MM\times{\cal P}(\MM)}K_{\mu,\partial}^\ell{\bf 1}(x)<1.
\end{equation}
Since $(x,\mu)\mapsto K_{\mu,\partial}^\ell{\bf 1}(x)$ is continuous on the compact space $\MM\times{\cal P}(\MM)$, it is enough to check that there exists $\ell$ such that for every $(x,\mu)\in \MM\times{\cal P}(\MM)$,  $K_{\mu,\partial}^\ell{\bf 1}(x)<1$. If it is not true, one can find a sequence $(x_n,\mu_n)$ such that $K_{\mu_n,\partial}^n{\bf 1}(x_n)=1$ for every $n$. Since $n\mapsto K_{\mu,\partial}^n{\bf 1}(x)$ is non-increasing, this implies that for every $(n,k)$ such that $k\le n$, $K_{\mu_n,\partial}^k{\bf 1}(x_n)=1$. Thus, for any limit point $(x^\star,\mu^\star)$ of $(x_n,\mu_n)_{n\ge1}$, $K_{\mu^\star,\partial}^k{\bf 1}(x^\star)=1$ for every $k\in\mathbb{N}$, which is a contradiction with the assumption of the lemma.
\end{proof}
\begin{proof}[Proof of \cref{lem:lip-prel}]
\noindent $(i)$ For every $n\in \mathbb{N}$, one sets
$$
 \kappa_n(\mu,\nu) = \sup_{\alpha \in {\cal P}(\MM)} \Vert \alpha \kmu^{n} - \alpha \knu^{n} \Vert_{TV} .
$$
We have $\kappa_0(\mu,\nu)=0$ and by \ref{condpdeux},
$$
\kappa_1(\mu,\nu) \le {c_2}  \Vert \mu- \nu \Vert_{TV}.
$$
Then, for every $n\ge0$,
\begin{align*}
\Vert \alpha\kmu^{n+1} - \alpha\knu^{n+1}\Vert_{TV}
&= \Vert \alpha\kmu \kmu^n - \alpha\knu \knu^n\Vert_{TV}\\
&\le \Vert \alpha\kmu \left[\kmu^n - \knu^n\right]\Vert_{TV}+\Vert \alpha\kmu \knu^n - \alpha\knu \knu^n\Vert_{TV}\\
&\le \kappa_n(\mu,\nu)+ \|\alpha\kmu-\alpha\knu\|_{TV}\\
&\le \kappa_n(\mu,\nu)+\kappa_1(\mu,\nu)\\
&\le \kappa_n(\mu,\nu)+c_2 \|\mu-\nu\|_{TV}\\
\end{align*}
An induction leads to
\begin{align*}
 \kappa_n(\mu,\nu)\le c_2 n\|\mu-\nu\|_{TV}.
\end{align*}
$(ii)$ First, since 
\begin{equation}\label{eq:deppsi}
R_\mu-R_\nu=\frac{1}{1-\varepsilon} (\kmu^\ell-\knu^\ell),
\end{equation}
the case $n=1$ follows from $(i)$.  Now, for $n\in \mathbb{N}$,
\begin{align*}
\sup_{\alpha\in{\cal P}(\MM)} \Vert \alpha R_\mu^{n+1} -  \alpha R_\nu^{n+1}  \Vert_{TV}&\le  \sup_{\alpha\in{\cal P}(\MM)} \Vert (\alpha R_\mu^{n}) R_\mu  -  (\alpha R_\mu^n) R_\nu  \Vert_{TV}\\
&+\sup_{\alpha\in{\cal P}(\MM)} \Vert (\alpha R_\mu^{n}) R_\nu - (\alpha  R_\nu^n) R_\nu  \Vert_{TV}\\
&\le \sup_{\alpha\in{\cal P}(\MM)} \Vert \alpha R_\mu  -  \alpha R_\nu  \Vert_{TV}+\sup_{\alpha\in{\cal P}(\MM)} \Vert \alpha R_\mu^{n}  - \alpha  R_\nu^n   \Vert_{TV}\\
&\le \frac{c_2\ell}{1-\varepsilon}\|\mu-\nu\|_{TV}+\sup_{\alpha\in{\cal P}(\MM)} \Vert \alpha R_\mu^{n}  - \alpha  R_\nu^n   \Vert_{TV}.
\end{align*}
The result follows by induction.
\end{proof}
\end{document}